\documentclass[a4paper,notitlepage,twoside,reqno,10pt]{amsart}

\usepackage{bbm,pifont,latexsym,dcolumn,indentfirst}
\usepackage[hypertexnames=false,bookmarksopen=true,linktocpage=true,pdfstartview={XYZ null null 1.25}]{hyperref}
\usepackage{amsthm,amsmath,amssymb,amscd,amsfonts,mathrsfs}
\usepackage{color,graphicx,xcolor,graphics,subfigure,extarrows,caption2}
\usepackage{pdfpages,titletoc}

\usepackage{autonum}

\newtheorem{thm}{Theorem}[section]
\newtheorem{lem}[thm]{Lemma}
\newtheorem{cor}[thm]{Corollary}
\newtheorem{prop}[thm]{Proposition}

\newtheorem{thmx}{Theorem}

\theoremstyle{definition}
\newtheorem*{defi}{Definition}
\newtheorem*{rmk}{Remark}

\newtheorem*{ques}{Question}

\theoremstyle{remark}

\newcommand{\EC}{\widehat{\mathbb{C}}}
\newcommand{\C}{\mathbb{C}}
\newcommand{\D}{\mathbb{D}}
\newcommand{\A}{\mathbb{A}}
\newcommand{\BH}{\mathbb{H}}

\newcommand{\R}{\mathbb{R}}
\newcommand{\T}{\mathbb{T}}
\newcommand{\Z}{\mathbb{Z}}

\newcommand{\MA}{\mathcal{A}}

\newcommand{\ii}{\textup{i}}
\newcommand{\re}{\textup{Re\,}}
\newcommand{\im}{\textup{Im\,}}

\newcommand{\Int}{\textup{Int}}
\newcommand{\Ext}{\textup{Ext}}

\newcommand{\Mod}{\textup{mod}}

\makeatletter\@addtoreset{equation}{section}\makeatother

\begin{document}

\author{FEI YANG}
\address{Department of Mathematics, Nanjing University, Nanjing 210093, P. R. China}
\email{yangfei@nju.edu.cn}

\title[Meromorphic functions with periodic Herman rings]{On the formulas of meromorphic functions with periodic Herman rings}

\begin{abstract}
We construct some explicit formulas of rational maps and transcendental meromorphic functions having Herman rings of period strictly larger than one. 
This gives an answer to a question raised by Shishikura in the 1980s. Moreover, the formulas of some rational maps with nested Herman rings are also found.

To obtain the formulas of transcendental meromorphic functions having periodic Herman rings, a crucial step is to find an explicit family of transcendental entire functions having bounded Siegel disks of all possible periods and rotation numbers. This is based on proving the existence of a Mandelbrot-like set of period one in the corresponding parameter space.
\end{abstract}



\date{\today}



\maketitle


\section{Introduction}\label{introduction}

Let $f$ be a rational map with degree at least two or a transcendental meromorphic function. In this paper by \textit{transcendental meromorphic} we mean the maps defined in the complex plane $\C$ with exactly one essential singularity at infinity.
A periodic Fatou component $U$ of $f$ is called a \textit{Herman ring} if $U$ is conformally isomorphic to an annulus and if the restriction of $f$, or some iterate of $f$ on $U$, is conformally conjugate to an irrational rotation of this annulus.
Unlike (super-)attracting basins, parabolic basins and Siegel disks, Herman rings have no associated periodic points. Hence it is usually difficult to determine whether a given map has a Herman ring.

There are two known methods for constructing Herman rings. Based on the work of Arnold, the original method is based on a careful analysis of the analytic diffeomorphism on the unit circle, due to Herman \cite{Her79}. He proved that every orientation-preserving analytic circle diffeomorphism of rotation number in a certain class of irrationals is analytically conjugate to the rigid rotation. Later Yoccoz generalized this class to the Herman type irrationals and gave an arithmetic characterization \cite{Yoc02}. In particular, according to Herman-Yoccoz, for any $a\in\C$ with $|a|>3$ and Herman type $\theta\in\R/\Z$, there exists $t=t(\theta)\in\R/\Z$ such that
\begin{equation}\label{equ:Blaschke}
B(z)= e^{2\pi\ii t}z^2\frac{z-a}{1-\overline{a}z}
\end{equation}
has a Herman ring containing the invariant unit circle of rotation number $\theta$.
An alternative method is based on quasiconformal surgery, by pasting together two Siegel disks, due to Shishikura \cite{Shi87}.
He also proved that if a rational map has a Herman ring, then the degree of this map is at least $3$. Hence the cubic Blaschke products in \eqref{equ:Blaschke} are the simplest kind of rational maps having Herman rings.

Examples of Herman rings for transcendental holomorphic maps $f:\C\setminus\{0\}\to\C\setminus\{0\}$ were constructed first by Herman in \cite[p.\,609]{Her85}, by studying the complex Arnold family
\begin{equation}\label{equ:Arnold}
A(z)=e^{2\pi\ii t}z e^{\frac{a}{2}(z-\frac{1}{z})},
\end{equation}
where $t,a\in(0,1)$. One may refer to \cite{BKL91b}, \cite{Fag99}, \cite{Gey01}, \cite{FG03}, \cite{FSV04}, \cite{BFGH05} and \cite{FH06} for the further study of Herman rings in this family.

For transcendental meromorphic functions $f:\C\to\EC$, the first example of Herman rings was constructed by Zheng in \cite{Zhe00}, and see \cite{DF04} for an independent work by Dom\'{i}nguez and Fagella.
The location of the poles and the omitted values plays an important role in the existence of Herman rings of transcendental meromorphic functions.

\medskip
The examples of Herman rings mentioned above are all invariant, i.e., they are Fatou components of period one. By quasiconformal surgery, Shishikura constructed cubic rational maps having Herman rings of any given period \cite{Shi87}, and Fagella and Peter proved the existence of Herman rings of any given period in transcendental meromorphic functions \cite{FP12}. In 1980s, Shishikura raised the following:

\begin{ques}[{Shishikura}]
Find an explicit form of a rational map which has a Herman ring of period two and draw its picture.
\end{ques}

This question was motivated by the study in \cite{Shi86,Shi87} (see \cite[p.\,4]{Shi14} for a formal statement).
In fact this question can be also asked for transcendental meromorphic functions.

\medskip
To find the explicit formulas of rational maps and transcendental meromorphic functions with Herman rings of high periods, besides the difficulty of the instability \cite{Man85},
one may observe another obstruction from the following result:

\begin{thmx}\label{thm:No-for-p2}
Any Herman ring of a meromorphic function $f$ (rational or transcendental) of period $p\geq 2$ cannot contain a circle $C$ satisfying $f^{\circ p}(C)=C$.
\end{thmx}

The main idea to prove Theorem \ref{thm:No-for-p2} is studying the dynamics by symmetric analytic extension. In \S\ref{sec:No-round-circle} we shall prove a stronger result for transcendental meromorphic functions: their Herman rings (of any period) cannot contain periodic circles.
From Theorem \ref{thm:No-for-p2} we know that, for meromorphic functions having $p$-periodic Herman rings with $p\geq 2$, it is hard to obtain the formulas by studying the circle diffeomorphism via following Herman and Yoccoz.

\medskip
In this paper we give an answer to Shishikura's question and draw a picture of the cycle of Herman rings of period $2$ (see Figure \ref{Fig:Herman} for an example).

\begin{thmx}\label{thm:rat-p}
Let $\theta$ be a Brjuno number and $p\geq 2$ an integer. There exists $a_0\geq 3$ such that for any $a>a_0$, there exist $b$ and $u=u(a,b)\in\C\setminus\{0\}$ such that
\begin{equation}\label{equ:f-rat}
f_{a,b}(z)=u z^2\frac{z-a}{1-az}+b
\end{equation}
has a $p$-cycle of Herman rings of rotation number $\theta$ and a super-attracting $p$-cycle  containing $0$.
In particular, if $p=2$, then $u=(ab-1)/(b(b-a))$.
\end{thmx}

\begin{figure}[!htpb]
  \setlength{\unitlength}{1mm}
  \centering
  \includegraphics[width=0.95\textwidth]{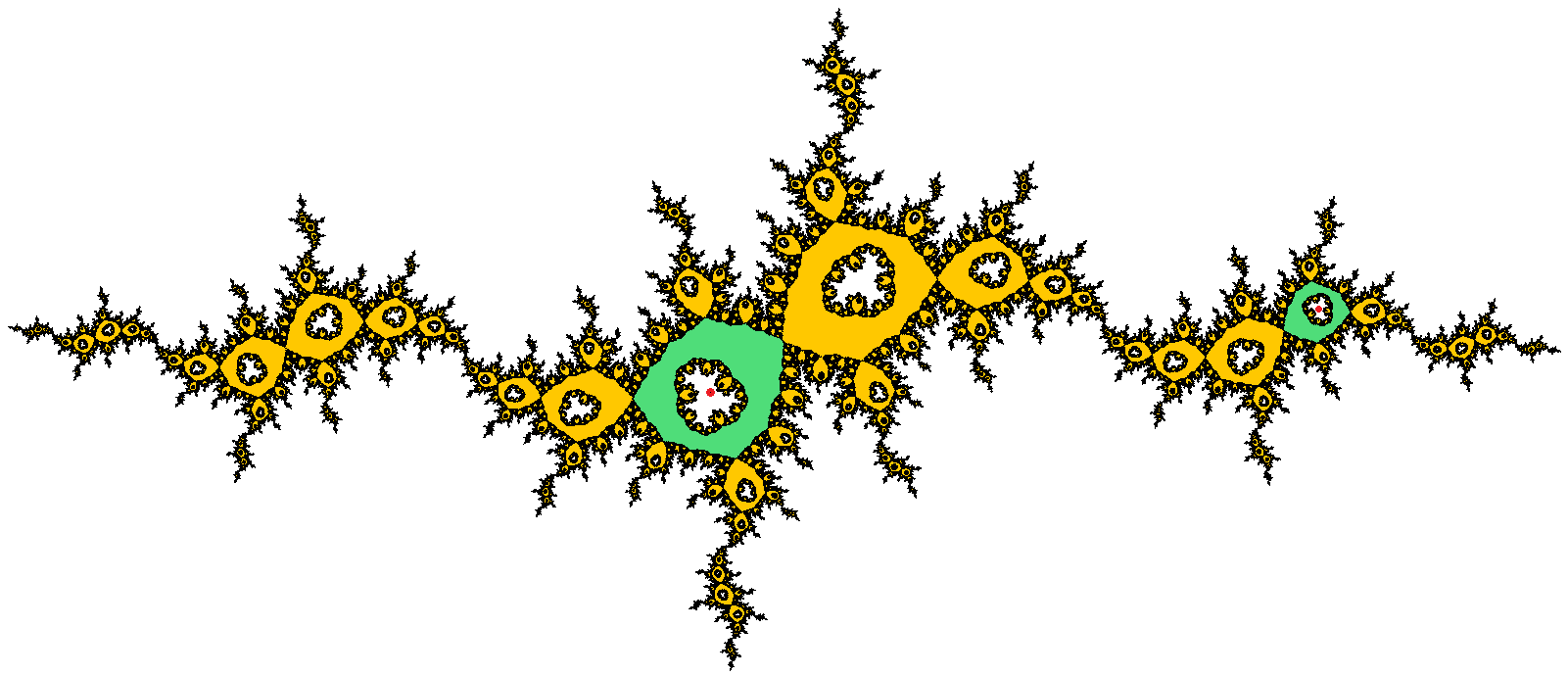}
  \caption{A rational map $f_{a,b}$ with a $2$-cycle of Herman rings (colored green), where $a=4$ and $b\approx11.03081483 -5.91931036\ii$ are chosen such that the rotation number is $\theta=(\sqrt{5}-1)/2$. The Julia set, preimages of the Herman rings and the super-attracting basins, are colored black, dark yellow and white respectively (this picture has been rotated). }
  \label{Fig:Herman}
\end{figure}

For the definition of Brjuno numbers, see \S\ref{sub:rat}.
For the proof of Theorem \ref{thm:rat-p}, we use quasiconformal surgery by deforming the complex structure in the Herman rings (see \S\ref{sec:non-nested-HR} and \cite[\S 6.1]{BF14}), and the key point is to show that the resulting rational maps after surgery have the form \eqref{equ:f-rat}.

\medskip
By a similar idea to Theorem \ref{thm:rat-p}, we can also find some explicit forms of transcendental meromorphic functions having Herman rings of any given period.

\begin{thmx}\label{thm:mero-p-FG}
Let $\theta$ be a Brjuno number and $p\geq 1$ an integer. Then there exist $a$, $b$ and $u=u(a,b)\in\C\setminus\{0\}$ such that
\begin{equation}\label{equ:g-mero-p}
g_{a,b}(z)=u\,\frac{z-b}{z-a}z^2 e^z
\end{equation}
has a $p$-cycle of Herman rings of rotation number $\theta$ and a super-attracting $p$-cycle which is different from $0$.
\end{thmx}

The proof of Theorem \ref{thm:mero-p-FG} is also based on quasiconformal surgery by turning periodic Siegel disks to Herman rings (see \S\ref{subsec:trans-I}).
Before performing the surgery, we need to choose a family of transcendental entire functions which provides the Siegel disks of any period and any Brjuno rotation number.
A natural candidate is the exponential family $z\mapsto \lambda e^z$ with $\lambda\in\C\setminus\{0\}$ (see \cite{RS09}). However, when the rotation number is of Herman type, the Siegel disks in this family are unbounded (see \cite{Her85}) and it is hard to draw their pictures even in the case when the period of the Siegel disks is one.

Another candidate is the sine family $z\mapsto \lambda\sin (z)$ with $\lambda\in\C\setminus\{0\}$. It has been observed that there exist copies of Mandelbrot set in the parameter space and one may obtain Siegel disks of any period and any Brjuno rotation number. However, for this family it is hard to obtain the specific formulas of the resulting transcendental meromorphic functions after surgery.

We also explain why we don't choose the family $z\mapsto \lambda ze^z$ in \S\ref{subsec:trans-I}.
In view of this, in this paper we consider the following family
\begin{equation}
E_\lambda(z)=\lambda z^2 e^z, \text{\quad where }\lambda\in\C\setminus\{0\},
\end{equation}
and prove that there exists a Mandelbrot-like set in the parameter plane such that the main cardioid corresponds to the hyperbolic component of period one.
By Douady-Hubbard's polynomial-like mapping theory, we conclude that there exist copies of Julia sets of quadratic polynomials with Siegel disks of all posssible periods and Brjuno rotation numbers in the family $E_\lambda$ (see \S\ref{sec:Siegel}).

In fact, so far we are not aware of a published proof of the existence of a family of \textit{explicit} transcendental entire functions which provides the copies of the Mandelbrot-like set such that they produce the Siegel disks of all possible periods and Brjuno rotation numbers, although computer experiments indicate that they exist indeed.

We would like to mention that, for \eqref{equ:g-mero-p} it is not easy to obtain the explicit values of $a$, $b$ and $u$ for a given period $p\geq 2$, since one needs to solve transcendental equations (see \S\ref{subsec:trans-I}). In view of this, as an easier example, we construct another explicit family of transcendental meromorphic functions such that they contain a cycle of Herman rings with period $2$ (see Theorem \ref{thm:mero-2} and Figure \ref{Fig:Herman-mero}).

\medskip
Let $\D$ be the unit disk whose boundary has been assigned the anticlockwise orientation. For an oriented Jordan curve $\gamma$ in $\EC$, there exists a conformal map which maps $\D$ onto a component $D_1$ of $\EC\setminus\gamma$ and respects the orientations of $\partial\D$ and $\gamma$. We call $D_1$ the \textit{interior} of $\gamma$ and denote by $\Int(\gamma)$. The other component $\EC\setminus\overline{D}_1$ is the \textit{exterior} of $\gamma$ and denoted by $\Ext(\gamma)$.
For an annulus $A\subset\EC$ with the oriented core curve $\gamma$, the orientation of $A$ is induced by $\gamma$. The interior $\Int(A)$ and exterior $\Ext(A)$ of $A$ is defined as the components of $\EC\setminus A$ contained in $\Int(\gamma)$ and $\Ext(\gamma)$ respectively.

Suppose $f$ has a cycle of Herman rings $\{A_1,\cdots,A_p\}$ with $p\geq 2$. Then these Herman rings can be oriented such that $f$ respects their orientations.
A cycle of Herman rings $\{A_1,\cdots,A_p\}$ is called \textit{non-nested} if either $A_j\subset \Ext (A_i)$, or $A_j\subset \Int (A_i)$ for any $1\leq i\neq j\leq p$.
From the proofs in \S\ref{sec:non-nested-HR}, one may see that the Herman rings constructed in Theorems \ref{thm:rat-p} and \ref{thm:mero-p-FG} are non-nested.
A $2$-cycle Herman rings $\{A_1,A_2\}$ is called \textit{nested} if either $A_1\subset \Int (A_2)$ and $A_2\subset \Ext (A_1)$, or $A_1\subset \Ext (A_2)$ and $A_2\subset \Int (A_1)$.

The existence of nested Herman rings for rational maps was proved by Shishikura in \cite{Shi87}, and for transcendental meromorphic functions by Fagella and Peter in \cite{FP12}.
In this paper we give the explicit formula of a family of rational maps having a $2$-cycle of nested Herman rings (see Figure \ref{Fig:Herman-inverse2} for an example).

\begin{thmx}\label{thm:period-2-inverse}
For any Brjuno number $\theta$, there exists $(r,t)\in(0,1)^2$ such that $h_{r,t}$ has a $2$-cycle of nested Herman rings of rotation number $\theta$, where
\begin{equation}\label{equ:h-r-t}
h_{r,t}(z)=e^{2\pi\ii t}\left(\frac{z-\frac{1}{r}}{1-\frac{z}{r}}\right)^3\,\frac{z-r^2e^{-2\pi\ii t}}{1-r^2e^{2\pi\ii t}z}.
\end{equation}
\end{thmx}

\begin{figure}[!htpb]
  \setlength{\unitlength}{1mm}
  \centering
  \includegraphics[width=0.46\textwidth]{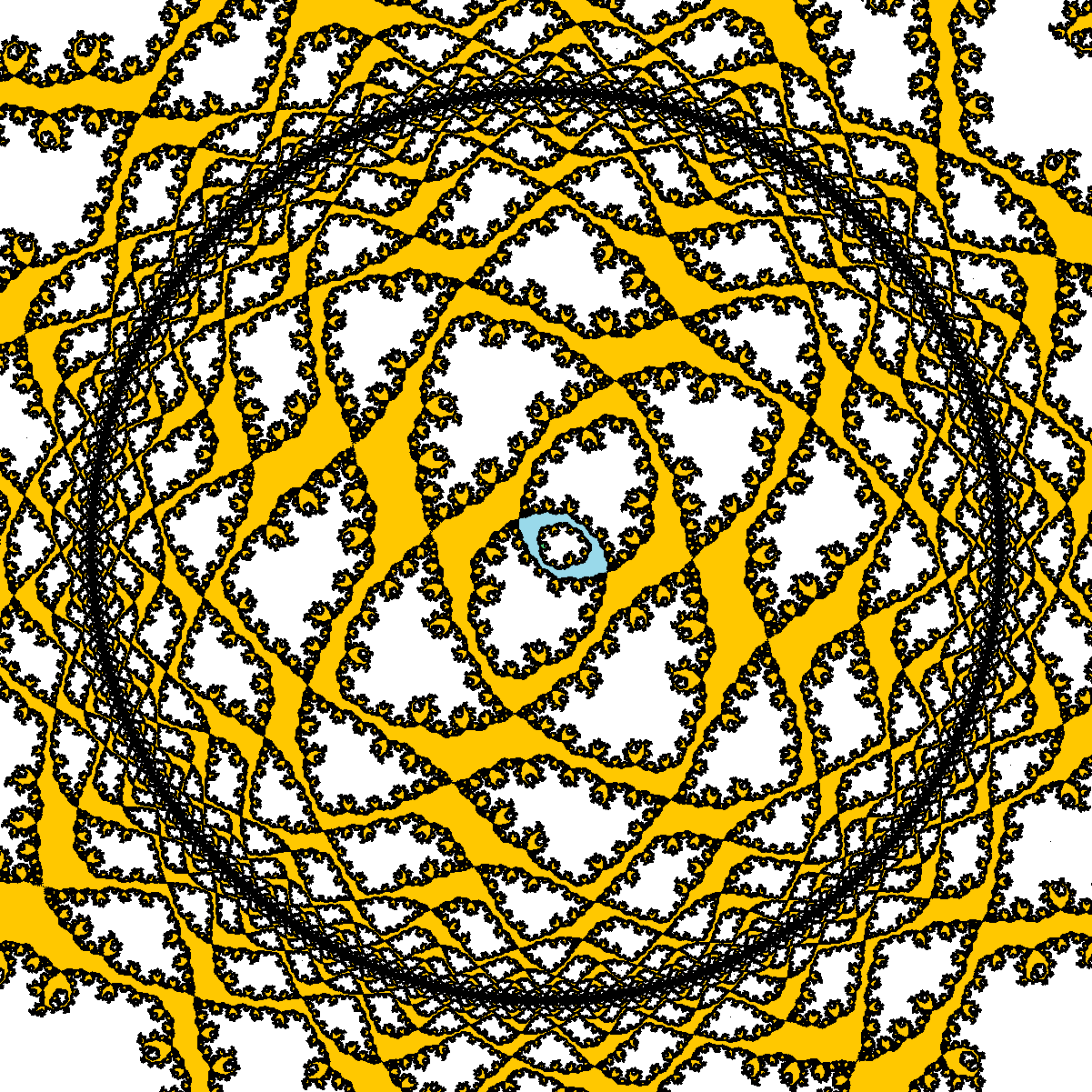}\qquad
  \includegraphics[width=0.46\textwidth]{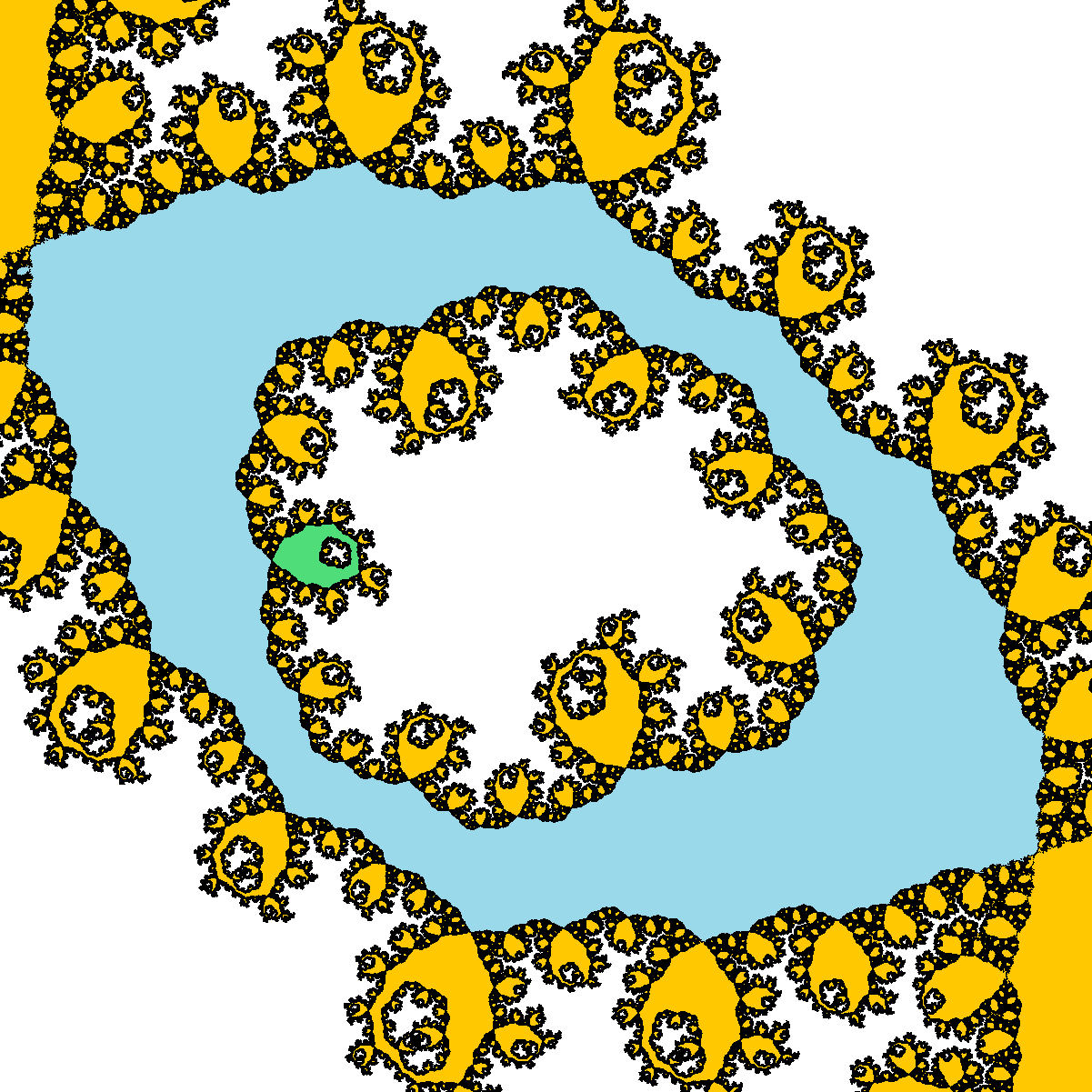}
  \caption{A quartic rational map $h_{r,t}$ with a $2$-cycle of nested Herman rings $A_1\leftrightarrow A_2$, where $r=\frac{1}{40}$ and $t\approx 0.34172383$. The picture on the right is the zoom of the left near one periodic Herman ring $A_1$ (colored green). The other periodic Herman ring $A_2$ is the symmetric image of $A_1$ about the unit circle. The ring $A_1$ in the left picture is too small to see. We color another preimage of $A_2$ light cyan so that one can roughly locate the position of $A_1$ in the left picture. }
  \label{Fig:Herman-inverse2}
\end{figure}

To the best of our knowledge, the only formula of a rational map with Herman rings of period $2$ appears in \cite{Shi86}, where a rough picture of the critical orbits was drawn. That example has the form $f(z)=z^2\frac{z-a}{z-b}+c$ and the three parameters $a$, $b$ and $c$ were found by numerical experiment (trial-and-error method).

To find the specific parameters in Figure \ref{Fig:Herman}, we first draw the parameter space (i.e., $b$-plane) of $f_{a,b}$ globally by fixing some $a$. Then we zoom near the bifurcation locus successively. One can catch the parameter corresponding to the periodic Herman rings of the given rotation number $\theta$ (This can be done similarly for generating Figure \ref{Fig:Herman-mero}). For Figure \ref{Fig:Herman-inverse2}, one can find the parameters in the real $2$-dimensional $(r,t)$-plane.

A cycle of Herman rings of high period may have various configurations (see \cite{Shi89}, \cite{Shi02} and \cite{FP12}). It is a challenging problem to find the explicit formulas for those complicated combinatorics. For various pictures of invariant Herman rings, one may refer to \cite{HR17} and Henriksen's homepage \cite{Hen14} (see also \cite{BBM18}). For the study of the existence of Herman rings of meromorphic functions, see \cite{Mil00b}, \cite{HK04}, \cite{Nay16}, \cite{Yan17}, \cite{HX19}, \cite{Roc20} and the references therein.

\section{Invariant circles versus periodic Herman rings}\label{sec:No-round-circle}

In this section, we prove the following result, which is stronger than Theorem \ref{thm:No-for-p2}.

\begin{thm}\label{thm:no-HR-revisited}
Let $A_1$ be a $p$-periodic Herman ring of $f$, where
\begin{enumerate}
\item $f$ is a rational map and $p\geq 2$; or
\item $f$ is a transcendental meromorphic function and $p\geq 1$.
\end{enumerate}
Then $A_1$ cannot contain any circle $C$ satisfying $f^{\circ p}(C)=C$.
\end{thm}

\begin{proof}
Suppose that there exists a circle $C$ in $A_1$ such that $f^{\circ p}(C)=C$. Without loss of generality, we assume that $C$ is the unit circle which separates the two boundary components of $A_1$.
Let $\tau(z)=1/\overline{z}$ be the inversion about the unit circle. Then $\tau(z)=\tau^{-1}(z)$. For a set $A$ in $\EC$, we denote $A^*=\tau(A)$.

\medskip
(a) Let $f$ be a rational map having a cycle of Herman rings $\{A_1,A_2,\cdots,A_p\}$, where $p\geq 2$, $A_{i+1}=f^{\circ i}(A_1)$ for $1\leq i\leq p-1$ and $A_1=f(A_p)$. By Schwarz symmetric extension, the map
\begin{equation}\label{equ:F}
F(z):=
\left\{
\begin{array}{ll}
f^{\circ p}(z)  &~~~~~~~\text{if}~z\in \EC\setminus\D, \\
\tau\circ f^{\circ p}\circ\tau^{-1}(z) &~~~~~~\text{if}~z\in \D
\end{array}
\right.
\end{equation}
is rational which coincides with $f^{\circ p}$ in $\EC\setminus\D$. Indeed, $\tau$ is the identity on the unit circle and $F$ is continuous. By the uniqueness of analytic functions, this implies that $\tau\circ f^{\circ p}\circ\tau^{-1}(z)=f^{\circ p}(z)$ for $z\in\D$.
Therefore, we have $\tau\circ f^{\circ p}\circ\tau^{-1}=f^{\circ p}$ on the whole Riemann sphere.

By definition, there exist an irrational number $\theta$, a number $0<r<1$ and a conformal map $\varphi:A_1\to\widetilde{\A}_r=\{\zeta\in\C:r<|\zeta|<\frac{1}{r}\}$ such that
\begin{equation}
\varphi\circ f^{\circ p}\circ\varphi^{-1}(\zeta)=R_\theta(\zeta):=e^{2\pi\ii\theta}\zeta, \text{\quad where }\zeta\in\widetilde{\A}_r.
\end{equation}
Note that $\tau\circ R_\theta\circ\tau^{-1}=R_\theta$ and $\tau\circ f^{\circ p}\circ\tau^{-1}=f^{\circ p}$. We have
\begin{equation}
(\tau\circ\varphi\circ\tau^{-1})\circ f^{\circ p}\circ(\tau\circ\varphi\circ\tau^{-1})^{-1}(\zeta)=e^{2\pi\ii\theta}\zeta, \text{\quad where }\zeta\in\tau(\widetilde{\A}_r)=\widetilde{\A}_r.
\end{equation}
Since $\tau\circ\varphi\circ\tau^{-1}:\tau(A_1)\to\widetilde{\A}_r$ is conformal, we have another cycle of Herman rings $\{A_1^*,A_2^*,\cdots,A_p^*\}$, where $A_{i+1}^*=f^{\circ i}(A_1^*)$ for $1\leq i\leq p-1$ and $A_1^*=f(A_p^*)$.
Note that $A_1$ contains the unit circle. Hence $A_1=A_1^*$. However, we have $f(A_1^*)=A_2^*=\tau(A_2)\neq A_2=f(A_1)$. This is a contradiction.

\medskip
(b) We divide the proof of this part into two cases: $p\geq 2$ and $p=1$. Suppose that $f$ is a transcendental meromorphic function having a cycle of Herman rings of period $p\geq 2$. Similar to \eqref{equ:F}, we define a new $F$ by replacing $\EC\setminus\D$ to $\C\setminus(\D\cup E)$ and $\D$ to $\D\setminus E$, where $E=\widetilde{E}\cup \widetilde{E}^*$ and $\widetilde{E}=\bigcup_{i=0}^{p-1} f^{-i}(\infty)$. Since $\C\setminus E$ is a domain, by the same reason as Part (a), we have $\tau\circ f^{\circ p}\circ\tau^{-1}=f^{\circ p}$ on $\C\setminus E$. Then the rest arguments are analogous.

Suppose that $f$ is a transcendental meromorphic function having a fixed Herman ring. Similar to \eqref{equ:F}, we define another $F$ by replacing $\EC\setminus\D$ to $\C\setminus\D$ and $\D$ to $\D\setminus \{0\}$. Then we have $\tau\circ f\circ\tau^{-1}=f$ on $\C\setminus \{0\}$. This implies that $0$ is an essential singularity, which contradicts the assumption that $f:\C\to\EC$ is meromorphic.
This finishes the proof of Theorem \ref{thm:no-HR-revisited} and hence Theorem \ref{thm:No-for-p2}.
\end{proof}

\begin{rmk}

(1) By Theorem \ref{thm:no-HR-revisited}, transcendental meromorphic functions cannot have  ``symmetric" Herman rings while rational maps and transcendental holomorphic functions $f:\C\setminus\{0\}\to\C\setminus\{0\}$ can.

(2) It is proved in \cite{BD98} that any transcendental holomorphic function $f:\C\setminus\{0\}\to\C\setminus\{0\}$ cannot have Herman rings of period strictly larger than one.
\end{rmk}

\section{Non-nested Herman rings}\label{sec:non-nested-HR}

In this section, we first state the deformation scheme of Herman rings which are motivated by \cite[\S6.1]{BF14}. Then we perform the quasiconformal surgery which transforms the cycles of Siegel disks to Herman rings (motivated by \cite{Shi87} and \cite{FP12}), and apply the deformation theory in the Herman rings to prove that there exist suitable parameters such that the resulting maps have the desired formulas in Theorems \ref{thm:rat-p} and \ref{thm:mero-p-FG} respectively. Finally, another family of transcendental meromorphic functions with $2$-cycle of Herman rings is obtained.

\subsection{Deformations of Herman rings}\label{subsec:DeformationRing}

If a meromorphic function $f$ has a cycle of Herman rings, then one can deform the complex structure in the rings in two ways: changing the conformal modulus and twisting the Herman rings. Combining these two means together, one can produce a holomorphic family of meromorphic functions having Herman rings. The following result is proved in the Main Theorem and Lemma 6.10 in \cite[\S 6.1]{BF14}.

\begin{thm}\label{thm:BF-6.1}
Let $f$ be a meromorphic function (rational or transcendental) having a Herman ring $A$ of modulus $m$. Assume that for all $k\geq 1$, $f^{\circ k}$ maps each component of $f^{-k}(A)$ to $A$ with degree $1$.
Then there exists an analytic family of meromorphic functions $f_\lambda$ parameterized by $\D\setminus\{0\}$, such that
\begin{enumerate}
\item $f$ is conformally conjugate to $f_{\lambda_0}$ with $\lambda_0=e^{-2\pi m}$;
\item $f_\lambda$ has a Herman ring of modulus $\frac{1}{2\pi}\log\frac{1}{|\lambda|}$ for all $\lambda\in\D\setminus\{0\}$; and
\item $f$ and $f_\lambda$ are quasiconformally conjugate and the conjugacy is conformal outside the grand orbit of $A$.
\end{enumerate}
\end{thm}

The original statement of Theorem \ref{thm:BF-6.1} is for rational maps. However, the proof is valid also for transcendental meromorphic functions since the quasiconformal surgery is only performed in the Herman rings. In the following subsections, we will apply Theorem \ref{thm:BF-6.1} to two special families.

\subsection{Formulas: the rational case}\label{sub:rat}

An irrational number $\theta\in(0,1)$ is called of \textit{Brjuno type} if its continued fraction expansion $[0;a_1,a_2,\cdots, a_n,\cdots]$ satisfies $\sum_{n\geq 1}\frac{\log q_{n+1}}{q_n}<+\infty$, where $q_n$ is the denominator of the simplified fraction $\frac{p_n}{q_n}=[0;a_1,a_2,\cdots, a_n]$. We use $\mathscr{B}$ to denote the set of all Brjuno numbers. It is known that every Diophantine number is Brjuno and that $\theta\in\mathscr{B}$ if and only if $1-\theta\in\mathscr{B}$. According to Siegel and Brjuno, every holomorphic germ $f(z)=e^{2\pi\ii\theta}z+O(z^2)$ is locally linearizable at $0$ if $\theta$ is of Brjuno. Yoccoz proved that the Brjuno condition is also necessary for the local linearization of quadratic polynomials \cite{Yoc95}.

For a Jordan curve $\gamma\subset\EC\setminus\{\infty\}$, we use $\Ext(\gamma)$ to denote the component of $\EC\setminus\gamma$ containing $\infty$ and use $\Int(\gamma)$ to denote the other. For an annulus $A\subset\EC\setminus\{\infty\}$, the exterior $\Ext(A)$ and interior $\Int(A)$ of $A$ are defined similarly.
Note that the definition of exterior and interior here is different from that in the introduction. 
We use $\partial_+A$ and $\partial_-A$ to denote the boundary components of $A$ contained in $\Ext(A)$ and $\Int(A)$ respectively.

\begin{lem}\label{lem-rat-pre}
For any given $\theta\in\mathscr{B}$ and integer $p\geq 2$, there exists an analytic family of cubic rational maps $f_\lambda$ parameterized by $\D\setminus\{0\}$:
\begin{equation}
f_\lambda(z)=u_\lambda z^2\frac{z-\alpha_\lambda}{1-\frac{2\alpha_\lambda-3}{\alpha_\lambda-2} z}+\beta_\lambda,
\end{equation}
where $\alpha_\lambda\in\C\setminus\big\{0,1,\frac{3}{2},2,3\big\}$ and $\beta_\lambda,u_\lambda\in\C\setminus\{0\}$ depend analytically on $\lambda$ such that for all $\lambda\in\D\setminus\{0\}$, the following statements hold:
\begin{enumerate}
\item $f_\lambda$ has a $p$-cycle of Herman rings $\{A_i^\lambda,1\leq i\leq p\}$ and a super-attracting $p$-cycle containing $0$;
\item $\Mod(A_1^\lambda)=\frac{1}{2\pi}\log\frac{1}{|\lambda|}$;
\item $f_\lambda$ and $f_{\lambda'}$ are quasiconformally conjugate and the conjugacy is conformal outside the grand orbit of $A_1^\lambda$ for all $\lambda,\lambda'\in\D\setminus\{0\}$; and
\item $A_1^\lambda$ separates $\{0,\frac{\alpha_\lambda(\alpha_\lambda-2)}{2\alpha_\lambda-3},\frac{\alpha_\lambda-2}{2\alpha_\lambda-3}\}$ from $\{1,\infty\}$, where $\{0,\frac{\alpha_\lambda(\alpha_\lambda-2)}{2\alpha_\lambda-3},1,\infty\}$ are the critical points of $f_\lambda$.
\end{enumerate}
\end{lem}

\begin{proof}
For any given $\theta\in\mathscr{B}$ and integer $p\geq 2$, there exists a quadratic polynomial $P(z)=z^2+\kappa_1 z+\kappa_0$ having a $p$-cycle of Siegel disks $\{\Delta_i:1\leq i\leq p\}$ of rotation number $\theta$ which contains the $p$-periodic point $0$, where $P(\Delta_i)=\Delta_{i+1}$ for $1\leq i\leq p-1$ and $0\in \Delta_1=P(\Delta_p)$.
Let $\Delta$ be the Siegel disk of $Q(z)=e^{-2\pi\ii\theta}z+z^2$ centered at the origin. There exist conformal mappings $\varphi:\Delta_1\to\D$ and $\phi:\Delta\to\D$ which conjugate $P^{\circ p}$ and $Q$ to $R_\theta(\zeta)=e^{2\pi\ii\theta}\zeta$ and $R_{-\theta}(\zeta)=e^{-2\pi\ii\theta}\zeta$ respectively.

\medskip
Given $0<r<1$, we denote $\gamma_1=\varphi^{-1}(\T_r)\subset\Delta_1$ and $\gamma_{i+1}=P^{\circ i}(\gamma_1)\subset\Delta_{i+1}$ for $1\leq i\leq p-1$, where $\T_r=\{\zeta\in\C:|\zeta|=r\}$ is the boundary of $\D_r=\{\zeta\in\C:|\zeta|<r\}$. For a small $\varepsilon>0$, we denote $\eta_\varepsilon(z):=\varepsilon/z$, then
\begin{equation}
U_\varepsilon:=\eta_\varepsilon\big(\EC\setminus\phi^{-1}(\overline{\D}_r)\big)
\end{equation}
is a small neighborhood of $0$, which is a Jordan domain. One can choose a sufficiently small $\varepsilon>0$ such that $\overline{U}_\varepsilon\subset\Int(\gamma_1)$ and the translations of $U_\varepsilon$ satisfy
\begin{equation}
\overline{U}_\varepsilon+P^{\circ i}(0)\subset\Int(\gamma_{i+1}), \text{\quad for all } 1\leq i\leq p-1.
\end{equation}
In particular, $A_{i+1}':=\Int(\gamma_{i+1})\setminus\big(\overline{U}_\varepsilon+P^{\circ i}(0)\big)$ is an annulus whose boundary consists of two real analytic Jordan curves, where $0\leq i\leq p-1$.

Define
\begin{equation}\label{equ:F-3.1}
F(z):=
\left\{
\begin{array}{ll}
P(z)  &~~~~~~~\text{if}~z\in \EC\setminus\bigcup_{i=1}^p\Int(\gamma_i), \\
\eta_\varepsilon\circ Q\circ\eta_\varepsilon^{-1}(z)+P(0) &~~~~~~\text{if}~z\in \overline{U}_\varepsilon, \\
z+P^{\circ (i+1)}(0)-P^{\circ i}(0) &~~~~~~\text{if}~z\in \overline{U}_\varepsilon+P^{\circ i}(0),\quad 1\leq i\leq p-1.
\end{array}
\right.
\end{equation}
By construction, $\varphi$ conjugates $F^{\circ p}=P^{\circ p}$ to the rigid rotation $R_\theta$ on $\partial_+A_1'$, and $\eta_\varepsilon\circ\phi\circ\eta_\varepsilon^{-1}$ also conjugates $F^{\circ p}=\eta_\varepsilon\circ Q\circ\eta_\varepsilon^{-1}$ to $R_\theta$ on $\partial_-A_1'$. To simplify the notations, we identify $A_{p+1}'$ to $A_1'$. There exist continuous maps
$\psi_i:\overline{A}_i'\to \overline{A}_{i+1}'$, where $1\leq i\leq p$, such that
\begin{itemize}
\item $\psi_i|_{\partial A_i'}=F|_{\partial A_i'}$ for $1\leq i\leq p$;
\item $\psi_i:A_i'\to A_{i+1}'$ is a quasiconformal mapping for $1\leq i\leq p$; and
\item $\psi:=\psi_p\circ\cdots\circ\psi_1:A_1'\to A_1'$ is conjugate to $R_\theta$ by a quasiconformal mapping $\xi:A_1'\to\widetilde{\A}_{r_0}$ for some $0<r_0<1$, where $\widetilde{\A}_{r_0}=\{\zeta\in\C:r_0<|\zeta|<\frac{1}{r_0}\}$.
\end{itemize}
We then define $F|_{A_i'}=\psi_i$ for $1\leq i\leq p$. Obviously, $F:\EC\to\EC$ is a quasiregular map which is analytic in $\EC\setminus\bigcup_{i=1}^p \overline{A}_i'$. See Figure \ref{Fig-surgery}.

\begin{figure}[!htpb]
  \setlength{\unitlength}{1mm}
  \centering
  \includegraphics[width=0.91\textwidth]{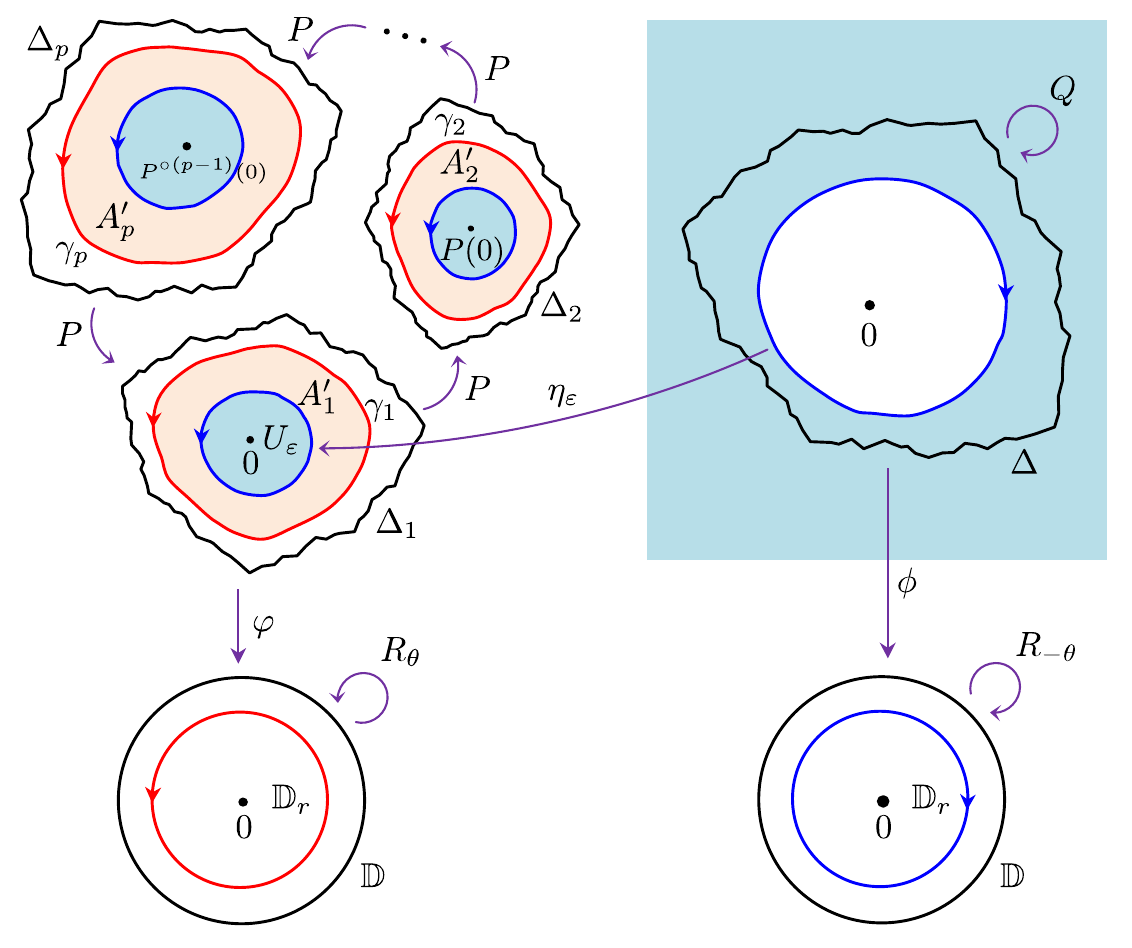}
  \caption{A sketch of the surgery construction.}
  \label{Fig-surgery}
\end{figure}

Now we pull back the standard structure $\sigma_0$ in $\widetilde{\A}_{r_0}$ to define an almost
complex structure $\sigma:=\xi^*(\sigma_0)$ in $A_1'$. Since $\xi:A_1'\to\widetilde{\A}_{r_0}$ conjugates $F^{\circ p}:A_1'\to A_1'$ to $R_\theta:\widetilde{\A}_{r_0}\to \widetilde{\A}_{r_0}$ and $\sigma_0$ is invariant under $R_\theta$, we obtain that $\sigma$ is invariant under $F^{\circ p}:A_1'\to A_1'$.  Pulling back $\sigma|_{A_1'}$ by defining $(F^{\circ i})^*(\sigma)$ in $A_{p+1-i}'$, where $1\leq i\leq p-1$, we obtain an almost complex structure $\sigma$ which is invariant in $\MA'=\bigcup_{i=1}^p A_i'$.

Next, we extend $\sigma$ to $F^{-1}(\MA')$ by letting $\sigma=F^*(\sigma)$. This does not redefine $\sigma$ in $\MA'$. Inductively, one can extend $\sigma$ to the grand orbit of $\MA'$ by letting $\sigma=(F^{\circ n})^*(\sigma)$ on $F^{-n}(\MA')$. At the rest place we define $\sigma=\sigma_0$. Since we pull back the ellipse field $\sigma$ in $\mathcal{A}'$ by holomorphic maps, it implies that $\sigma$ has uniformly bounded dilatation. By the measurable Riemann mapping theorem, there exists a quasiconformal mapping $\Phi:\EC\to\EC$ integrating the almost complex structure $\sigma$ such that $\Phi^*\sigma_0=\sigma$ and $\Phi:(\EC,\sigma)\rightarrow (\EC,\sigma_0)$ is an analytic isomorphism. Hence $f:=\Phi\circ F\circ \Phi^{-1}:(\EC,\sigma_0)\to(\EC,\sigma_0)$ is a rational map.

Let $c_1$ be the unique critical point of $P$ in $\C$. We assume that the quasiconformal mapping $\Phi:\EC\to\EC$ is normalized by $\Phi(\infty)=\infty$, $\Phi(0)=0$ and $\Phi(c_1)=1$.
Let $c_2$ be the unique critical point of $Q$ in $\C$. Then $F$ has four different branching points $\infty$, $0$, $c_1$ and $\eta_\varepsilon(c_2)$. Since the branching point $0$ is $p$-periodic for $F$, it follows that $f$ is a cubic rational map having a super-attracting $p$-cycle containing the critical point $0$, a super-attracting fixed point at $\infty$ and two different critical points $1$ and $c:=\Phi(\eta_\varepsilon(c_2))$ in $\C$.
A direct calculation shows that $f$ has the formula
\begin{equation}
f(z)=u z^2\frac{z-\alpha}{1-\omega z}+\beta,
\end{equation}
where $\omega=\frac{2\alpha-3}{\alpha-2}$, $\alpha\in\C\setminus\big\{0,1,\frac{3}{2},2,3\big\}$ and $\beta ,u \in\C\setminus\{0\}$ satisfy $f^{\circ p}(0)=0$.
In particular, $f$ has four simple critical points $0$, $1$, $\infty$ and $c=\frac{\alpha(\alpha-2)}{2\alpha-3}$. By the surgery construction, $f$ has a $p$-cycle of Herman rings $\{A_i:1\leq i\leq p\}$, where $\Phi(A_i')\subset A_i$. Moreover, the Herman ring $A_1$ separates $\{0,c,1/\omega\}$ from $\{1,\infty\}$, where $1/\omega=c/\alpha$ is the unique pole of $f$ in $\C$.

\medskip
From the proof of Theorem \ref{thm:BF-6.1} (see \cite[\S 6.1]{BF14}), one obtains an analytic family of cubic rational maps $f_\lambda=\Phi_\lambda\circ f\circ \Phi_\lambda^{-1}$ parameterized by $\D\setminus\{0\}$, where $\Phi_\lambda:\EC\to\EC$ is an analytic family of quasiconformal mappings fixing $0$, $1$ and $\infty$ which is parameterized by $\D\setminus\{0\}$. Note that $f$ has exactly four critical points $\{0,1,\infty, c\}$. This implies that $f_\lambda$ has four critical points $\{0,1,\infty,c_\lambda=\Phi_\lambda(c)\}$ and it can be written as:
\begin{equation}
f_\lambda(z)=u_\lambda z^2\frac{z-\alpha_\lambda}{1-\omega_\lambda z}+\beta_\lambda,
\end{equation}
where $\omega_\lambda=\frac{2\alpha_\lambda-3}{\alpha_\lambda-2}$, $\alpha_\lambda\in\C\setminus\big\{0,1,\frac{3}{2},2,3\big\}$ and $\beta_\lambda,u_\lambda\in\C\setminus\{0\}$ satisfy $f_\lambda^{\circ p}(0)=0$. Moreover, $c_\lambda=\frac{\alpha_\lambda(\alpha_\lambda-2)}{2\alpha_\lambda-3}=\alpha_\lambda/\omega_\lambda$. It suffices to prove that the coefficients $\alpha_\lambda$, $\beta_\lambda$ and $u_\lambda$ depend analytically on $\lambda$.

Since $1/\omega=c/\alpha$ is the unique pole of $f$ in $\C$, it follows that $1/\omega_\lambda=\Phi_\lambda(1/w)$ is the unique pole of $f_\lambda$ and $\lambda\mapsto \omega_\lambda$ is analytic.
For fixed $z\in\C\setminus\{1/\omega\}$, the map $\lambda\mapsto f_\lambda(z)$ is analytic. Hence $\beta_\lambda=f_\lambda(0)$, $c_\lambda=\Phi_\lambda(c)$ and $\alpha_\lambda=c_\lambda\omega_\lambda$ depend analytically on $\lambda$. Finally, $\lambda\mapsto u_\lambda=\frac{1-\omega_\lambda}{1-\alpha_\lambda}\big(f_\lambda(1)-\beta_\lambda\big)$ is also analytic.
The proof is complete.
\end{proof}

\begin{rmk}
The principle of the surgery in the proof of Lemma \ref{lem-rat-pre} is similar to \cite[\S 9]{Shi87}: by pasting two cycles of Siegel disks together to obtain a cycle of Herman rings. But the method of constructing quasiregular map is slightly different. We paste the dynamics of $Q$ in the $p$-cycle of Siegel disks of $P$ by Euclidean translations while Shishikura pasted $p$ Riemann hemispheres along the periodic curves in the $p$-cycle of Siegel disks.
\end{rmk}

\begin{proof}[Proof of Theorem \ref{thm:rat-p}]
By Lemma \ref{lem-rat-pre}(d), $A_1^\lambda$ is an annulus separating $\{0,\frac{\alpha_\lambda(\alpha_\lambda-2)}{2\alpha_\lambda-3}$, $\frac{\alpha_\lambda-2}{2\alpha_\lambda-3}\}$ from $\{1,\infty\}$. By Lemma \ref{lem-rat-pre}(b), the modulus of $A_1^\lambda$ tends to infinity as $\lambda$ tends to zero.
According to Teichm\"{u}ller's modulus extremal theorem \cite[Theorem 4.7, p.\,72]{Ahl10}, this implies that
$$\lim_{\lambda\to 0}\alpha_\lambda=2.$$
Note that we have a holomorphic map $\alpha_\lambda:\D\setminus\{0\}\to\C\setminus\{0,1,\frac{3}{2},2,3\}$.
Therefore, $\alpha_\lambda$ can be extended to a holomorphic map $\widetilde{\alpha}_\lambda:\D\to\C\setminus\{0,1,\frac{3}{2},3\}$ with $\widetilde{\alpha}_\lambda(0)=2$ by removable singularity theorem.
By the openness of $\widetilde{\alpha}_\lambda$, there exists a number $0<\delta\leq 1$ such that $\widetilde{\alpha}_\lambda(\D)$ contains the segment $[2,2+\delta)$. In particular, $\alpha_\lambda(\D)$ contains the open interval $(2,2+\delta)$.

For $\alpha_\lambda\in(2,2+\delta)$, we denote $L_\lambda(z)=\kappa_\lambda z$, where $\kappa_\lambda=\sqrt{\frac{2\alpha_\lambda-3}{\alpha_\lambda(\alpha_\lambda-2)}}>1$. Then
\begin{equation}
L_\lambda\circ f_\lambda\circ L_\lambda^{-1}(z)=\frac{u_\lambda}{\kappa_\lambda^2} z^2\frac{z-\alpha_\lambda\kappa_\lambda}{1-\alpha_\lambda\kappa_\lambda z}
+\kappa_\lambda \beta_\lambda.
\end{equation}
Set $a=\alpha_\lambda\kappa_\lambda$, $u=u_\lambda/\kappa_\lambda^2$ and $b=\kappa_\lambda \beta_\lambda$. We have
\begin{equation}
a^2=\frac{\alpha_\lambda(2\alpha_\lambda-3)}{\alpha_\lambda-2}=5+2\Big((\alpha_\lambda-2)+\frac{1}{\alpha_\lambda-2}\Big)>5+2(\delta+\delta^{-1}).
\end{equation}
The first statement of Theorem \ref{thm:rat-p} holds if we set $a_0:=\sqrt{5+2(\delta+\delta^{-1})}\geq 3$.

If $p=2$, by solving $f_{a,b}^{\circ 2}(0)=0$, which is equivalent to $f_{a,b}(b)=0$, we have $u=(a b-1)/(b(b-a))$.
\end{proof}

\begin{rmk}
Similar to \cite[Theorem B]{FG03}, one can prove that there is no uniform bound of $a_0$ for all $\theta\in\mathscr{B}$. We believe that in Theorem \ref{thm:rat-p} one can choose $a_0=3$ for all Diophantine (even Herman) numbers $\theta$, which is supported by computer experiments.
\end{rmk}

\subsection{Formulas: the transcendental case I}\label{subsec:trans-I}

Before performing the surgery which turns the periodic Siegel disks to Herman rings, we need to know the existence of $p$-cycle of Siegel disks in some specific transcendental entire functions.

\begin{prop}\label{prop:Siegel-p}
For any $\theta\in\mathscr{B}$ and $p\geq 1$, there exists $\lambda\in\C\setminus\{0\}$ such that $E_\lambda(z)=\lambda z^2 e^z$ has a $p$-cycle of Siegel disks with rotation number $\theta$.
\end{prop}

Fagella proved that there exists a baby Mandelbrot set in the parameter space of $G_\lambda(z)=\lambda z e^z$ \cite{Fag95}. However, the corresponding main cardioid is the hyperbolic component of period two. Therefore, her result implies that for any given $p\geq 1$, there exists $\lambda$ such that $G_\lambda$ contains a $2p$-cycle of Siegel disks of any given rotation number in $\mathscr{B}$. For this family, Kremer proved that there exists a hyperbolic component  in $\C\setminus\overline{\D}$ attached at $e^{2\pi\ii\frac{q}{p}}$ for all rational numbers $\frac{q}{p}$ with $(p,q)=1$ \cite{Kre01}. However, it is not known whether these hyperbolic components are bounded and one cannot conclude that the boundary of these hyperbolic components contain all types of Siegel parameters.

Katagata proved the existence of transcendental entire functions having quadratic Siegel disks of any given period and any rotation number. However, the formulas of those functions are not known, they contain wandering domains and hard to be written as an analytic family \cite{Kat17}.

The family in Proposition \ref{prop:Siegel-p} was studied first by Fagella and Garijo in \cite{FG03a} and \cite{FG07} in a general form $z\mapsto\lambda z^m e^z$ with $m\geq 2$, and the main study objects in the parameter planes are the capture zones, but not the renormalizable parameters near the free critical point.
We delay the proof of Proposition \ref{prop:Siegel-p} to \S\ref{sec:Siegel}.

\begin{proof}[Proof of Theorem \ref{thm:mero-p-FG} assuming Proposition \ref{prop:Siegel-p}]
We only give a sketch of the proof since the idea is similar to Lemma \ref{lem-rat-pre}.
For any given $\theta\in\mathscr{B}$ and $p\geq 2$, by Proposition \ref{prop:Siegel-p}, we choose $\lambda\in\C\setminus\{0\}$ such that $E_\lambda(z)=\lambda z^2 e^z$ has a $p$-cycle of Siegel disks $\{\Delta_i:1\leq i\leq p\}$ of rotation number $\theta$ containing the $p$-periodic point $s_0\neq 0$, where $E_\lambda(\Delta_i)=\Delta_{i+1}$ for $1\leq i\leq p-1$ and $s_0\in \Delta_1=E_\lambda(\Delta_p)$.
Let $\Delta$ be the Siegel disk of $Q(z)=e^{-2\pi\ii\theta}z+z^2$.

Similar to Lemma \ref{lem-rat-pre}, we paste the exterior of an invariant proper subdisk of $\Delta$ to a small neighborhood of $s_0$ in $\Delta_1$. Then a series of Euclidean translations are added to the orbit of a small neighborhood of $s_0$. By adopting the same notations as in Lemma \ref{lem-rat-pre} (see also Figure \ref{Fig-surgery}), the pre-quasiregular map $G$ is defined as
\begin{equation}
G(z):=
\left\{
\begin{array}{ll}
E_\lambda(z)  &~~~~~~~\text{if}~z\in \EC\setminus\bigcup_{i=1}^p\Int(\gamma_i), \\
\eta_\varepsilon\circ Q\circ\eta_\varepsilon^{-1}(z)+E_\lambda(s_0) &~~~~~~\text{if}~z\in \overline{U}_\varepsilon, \\
z+E_\lambda^{\circ (i+1)}(s_0)-E_\lambda^{\circ i}(s_0) &~~~~~~\text{if}~z\in \overline{U}_\varepsilon+E_\lambda^{\circ i}(s_0),\quad 1\leq i\leq p-1.
\end{array}
\right.
\end{equation}
After designating $p$ quasiconformal mappings $\psi_i$, where $1\leq i\leq p$, in the corresponding annuli $A_i'$ respectively, such that their composition $\psi=\psi_p\circ\cdots\circ\psi_1$ is quasiconformally conjugate to the irrational rotation $R_\theta$, one can obtain a quasiregular map $G:\C\to\EC$, which admits an invariant ellipse field.

There exists a quasiconformal mapping $\Phi:\EC\to\EC$ which fixes $0$, $s_0$ and $\infty$, such that $g=\Phi\circ G\circ\Phi^{-1}:\C\to\EC$ is a meromorphic function having the following properties:
\begin{itemize}
\item[(i)] $g$ is holomorphic in $\C$ except it has exactly one simple pole $a\in\C\setminus\{0\}$;
\item[(ii)] $g$ has a super-attracting fixed point $0$ with local degree $2$ and a super-attracting $p$-cycle containing $s_0$;
\item[(iii)] $0$ is the unique asymptotic value and $0$ has exactly three preimages (counting with multiplicity); and
\item[(iv)] $g$ has a $p$-cycle of Herman rings $\{A_i:1\leq i\leq p\}$.
\end{itemize}
By (i) and (ii), $g$ has the following form
\begin{equation}
g(z)=\frac{z^2}{z-a}\, g_1(z),
\end{equation}
where $g_1:\C\to\C$ is a transcendental entire function. By (iii), $g$ can be written as
\begin{equation}
g(z)=\frac{(z-b)z^2}{z-a}\, g_2(z),
\end{equation}
where $b\in\C\setminus\{a\}$ and $g_2:\C\to\C$ is a non-constant transcendental entire function which avoids the origin. Hence $g_2$ can be written as
 $g_2(z)=e^{\omega(z)}$, where $\omega:\C\to\C$ is a non-constant entire function. Since the map $E_\lambda$ around $\infty$ has not been modified during the surgery, the order of $g$ is one and $\omega(z)$ is linear. By replacing new parameters $a$ and $b$ up to linear conjugacy if necessary, one can write $g$ as \eqref{equ:g-mero-p}.
\end{proof}

According to the proof of Theorem \ref{thm:mero-p-FG}, by a direct calculation, it follows that there exist $a$, $b\in\C\setminus\{0\}$
such that the function
\begin{equation}
g(z)=\frac{a-1}{b-1}\,\frac{z-b}{z-a}\,z^2e^{\xi (z-1)}, \text{\quad where } \xi=\frac{2a-ab-1}{(a-1)(b-1)},
\end{equation}
has a Herman ring of period one and a super-attracting fixed point at $1$.
But it is hard to obtain the formula of $u$ in $a$ and $b$ when $p\geq 2$.

\subsection{Formulas: the transcendental case II}\label{subsec:trans-II}

In view of the difficulty to generate the explicit parameters of $g_{a,b}$ with Herman rings of period at least $2$, in this subsection we provide another family of entire functions with a $2$-cycle of Siegel disks, which will be used for generating the explicit formula of transcendental meromorphic functions with Herman rings of period two.

\begin{lem}\label{lem:Siegel-2}
For any $\theta\in\mathscr{B}$, there exists $b\in\C\setminus\{0\}$ such that $E(z)=-z e^{z-b}+b$ has a $2$-periodic Siegel point $0$ whose rotation number is $\theta$.
\end{lem}

\begin{proof}
Obviously, $E$ has a $2$-cycle $\{0,b\}$. By a direct calculation, we have
\begin{equation}
E'(0)E'(b)=(b+1)e^{-b}.
\end{equation}
Therefore, it is sufficient to prove that the equation $(b+1)e^{-b}=e^{2\pi\ii\theta}$ has at least one solution $b$ in $\C\setminus\{0\}$. We claim that $z\mapsto (z+1)e^{-z}$ is a surjection from the complex plane to itself. Otherwise, there exists $\xi\in\C$  and an entire function $\chi$ satisfying $(z+1)e^{-z}-\xi=e^{\chi(z)}$. The left-hand equation takes $-\xi$ only once, while the right-hand takes $-\xi$ infinitely times. Note that $b\neq 0$ since $\theta$ is irrational. This contradiction concludes the lemma.
\end{proof}

\begin{thm}\label{thm:mero-2}
For any $\theta\in\mathscr{B}$, there exists $a_1\in(0,1]$ such that for any $0<a<a_1$, there exists $b\in\C\setminus\{0\}$ such that
\begin{equation}\label{equ:f-mero}
m_{a,b}(z)=\frac{a-b}{be^b}\,\frac{z^2}{z-a}\,e^z+b
\end{equation}
has a $2$-cycle of Herman rings of rotation number $\theta$.
\end{thm}

See Figure \ref{Fig:Herman-mero} for an example.

\begin{figure}[!htpb]
  \setlength{\unitlength}{1mm}
  \centering
  \includegraphics[width=0.98\textwidth]{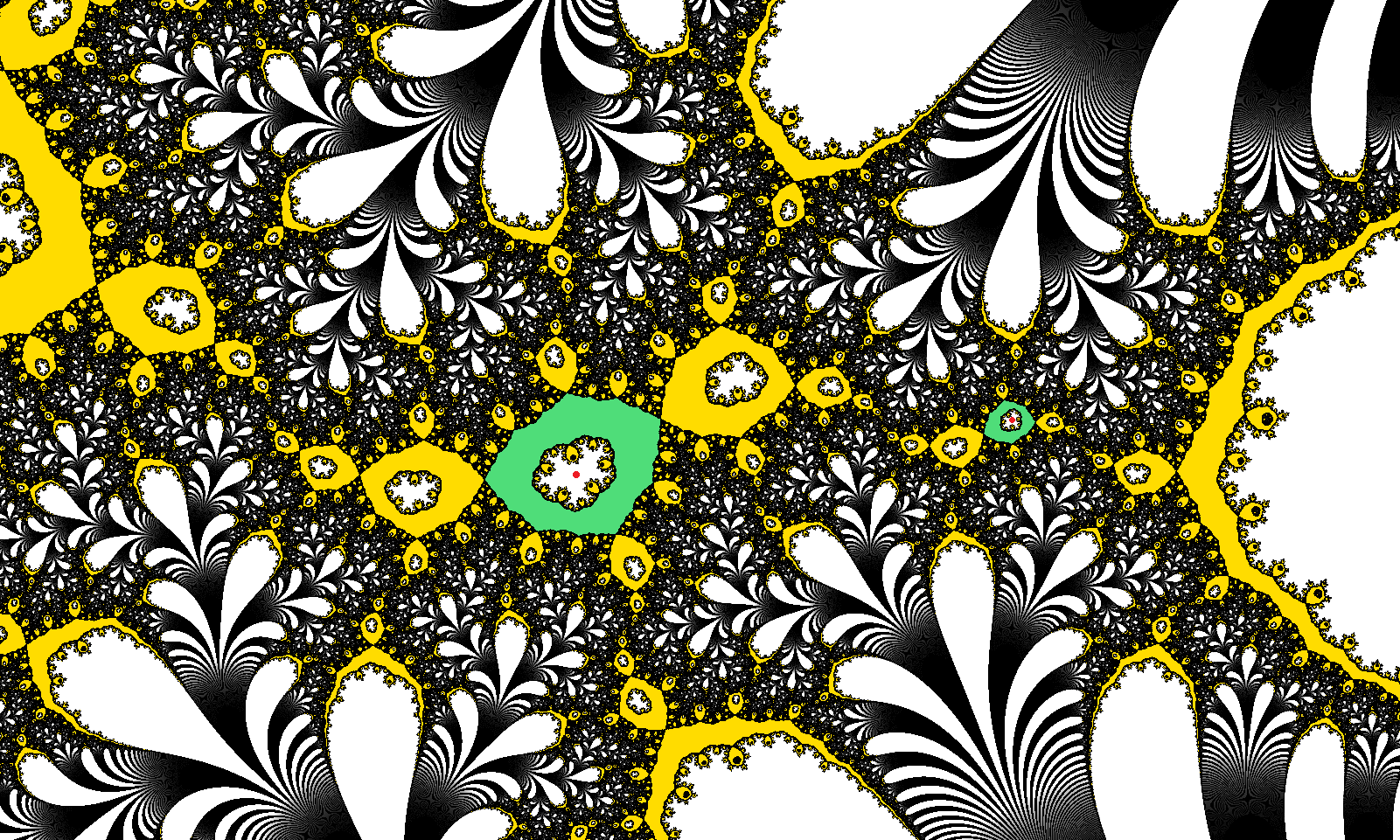}
  \caption{A transcendental meromorphic function $m_{a,b}$ with a $2$-cycle of Herman rings (colored green), where $a=\frac{1}{100}$ and $b\approx-1.23796766-0.16535887\ii$ are chosen such that the rotation number is $\theta=(\sqrt{5}-1)/2$. The Julia set, preimages of the Herman rings and the super-attracting basins, are colored black, dark yellow and white respectively.  Figure range: $[-2.9,1.1]\times[-1.2,1.2]$.}
  \label{Fig:Herman-mero}
\end{figure}

\begin{proof}
For any given $\theta\in\mathscr{B}$, by Lemma \ref{lem:Siegel-2}, let $b\in\C\setminus\{0\}$ be a number such that $E(z)=-z e^{z-b}+b$ has a $2$-cycle of Siegel disks $\{\Delta_1,\Delta_2\}$ of rotation number $\theta$ containing the $2$-periodic point $0$, where $0\in \Delta_1$.
Let $\Delta$ be the Siegel disk of $Q(z)=e^{-2\pi\ii\theta}z+z^2$.

Similar to Lemma \ref{lem-rat-pre}, we paste the exterior of an invariant proper subdisk of $\Delta$ to a small neighborhood of $0$ in $\Delta_1$. Then an Euclidean translation is added in a small neighborhood of $b$. By adopting the same notations as in Lemma \ref{lem-rat-pre} (see also Figure \ref{Fig-surgery}), the pre-quasiregular map $M$ is defined as
\begin{equation}
M(z):=
\left\{
\begin{array}{ll}
E(z)  &~~~~~~~\text{if}~z\in \C\setminus\big(\Int(\gamma_1)\cup\Int(\gamma_2)\big), \\
\eta_\varepsilon\circ Q\circ\eta_\varepsilon^{-1}(z)+b &~~~~~~\text{if}~z\in \overline{U}_\varepsilon, \\
z-b &~~~~~~\text{if}~z\in \overline{U}_\varepsilon+b.
\end{array}
\right.
\end{equation}
After designating two quasiconformal mappings $\psi_1$ and $\psi_2$ in two corresponding annuli $A_1'$ and $A_2'$ respectively such that their composition $\psi=\psi_2\circ\psi_1$ is quasiconformally conjugate to the irrational rotation $R_\theta$, one can obtain a quasiregular map $M:\C\to\EC$, which admits an invariant ellipse field.

There exists a quasiconformal mapping $\Phi:\EC\to\EC$ which fixes $0$, $b$ and $\infty$, such that $m=\Phi\circ M\circ\Phi^{-1}:\C\to\EC$ is a meromorphic function having the following properties:
\begin{itemize}
\item[(i)] $m$ is holomorphic in $\C$ except it has exactly one simple pole $a\in\C\setminus\{0\}$;
\item[(ii)] $m$ has a super-attracting $2$-cycle $\{0,b\}$, where $0$ is a simple critical point, $b$ is the unique asymptotic value and $m^{-1}(b)=\{0\}$;
\item[(iii)] $m$ has exactly $3$ different critical points $0$, $c_1$ and $c_2$; and
\item[(iv)] $m$ has a $2$-cycle of Herman rings $\{A_1,A_2\}$, where $A_1$ separates $\{0,a,c_1\}$ from $\{c_2,b,\infty\}$.
\end{itemize}
By (i), $m$ has the following form
\begin{equation}
m(z)=\frac{1}{z-a}\, m_1(z),
\end{equation}
where $m_1:\C\to\C$ is a transcendental entire function.
By (ii), since $m^{-1}(b)=\{0\}$ and $0$ is a simple critical point of $m$, it follows that $m$ can be written as
\begin{equation}
m(z)=\frac{z^2}{z-a}\, m_2(z)+b,
\end{equation}
where $m_2:\C\to\C$ is a non-constant transcendental entire function which avoids the origin. Hence $m_2$ can be written as
$m_2(z)=e^{\omega(z)}$, where $\omega:\C\to\C$ is a non-constant entire function.
Since we have not modified $E$ around the infinity, the order of $m$ is one and hence $\omega$ is linear. Note that $m$ has a $2$-cycle $\{0,b\}$. By replacing new constants $a$ and $b$ if necessary, it follows that one can write $m$ as
\begin{equation}
m(z)=\frac{a-b}{be^b}\,\frac{z^2}{z-a}\,e^z+b.
\end{equation}

From the proof of Theorem \ref{thm:BF-6.1} (see \cite[\S 6.1]{BF14}), one obtains an analytic family of transcendental meromorphic maps $m_\lambda=\Phi_\lambda\circ m\circ \Phi_\lambda^{-1}$ parameterized by $\D\setminus\{0\}$, where $\Phi_\lambda:\EC\to\EC$ is an analytic family of quasiconformal mappings fixing $0$, $1$ and $\infty$ which is parameterized by $\D\setminus\{0\}$. Let $a_\lambda=\Phi_\lambda(a)$ be the unique pole of $m_\lambda$ in $\C$ and denote $b_\lambda=f_\lambda(0)$. Since $m_\lambda$ has a super-attracting $2$-cycle $\{0,b_\lambda\}$, it follows that $m_\lambda$ can be written as
\begin{equation}
m_\lambda(z)=\frac{a_\lambda-b_\lambda}{b_\lambda e^{b_\lambda}}\frac{z^2}{z-a_\lambda}\,e^z+b_\lambda.
\end{equation}
Moreover, $m_\lambda$ has the following properties for all $\lambda\in\D\setminus\{0\}$:
\begin{itemize}
\item $m_\lambda$ has a $2$-cycle of Herman rings $\{A_1^\lambda, A_2^\lambda\}$;
\item $\Mod(A_1^\lambda)=\frac{1}{2\pi}\log\frac{1}{|\lambda|}$; and
\item $A_1^\lambda$ separates $\{0,a_\lambda,c_1^\lambda\}$ from $\{c_2^\lambda,\infty\}$, where $c_1^\lambda$ and $c_2^\lambda$ are the critical points of $m_\lambda$ having the form:
\begin{equation}
\big\{c_1^\lambda,c_2^\lambda\big\}=\left\{c_\lambda^\pm=\frac{a_\lambda-1\pm\sqrt{1+6a_\lambda+a_\lambda^2}}{2}\right\},
\end{equation}
where we adopt the convention that $\arg\sqrt{z}\in(-\frac{\pi}{2},\frac{\pi}{2}]$ for $z\in\C\setminus\{0\}$.
\end{itemize}

Since $c_1^\lambda\neq c_2^\lambda$, we have $a_\lambda\in\C\setminus\{0,-3\pm 2\sqrt{2}\}$. Hence $a_\lambda=\Phi_\lambda(a):\D\setminus\{0\}\to\C\setminus\{0,-3\pm 2\sqrt{2}\}$ is holomorphic.
By Lemma \ref{lem-rat-pre}(b), the modulus of $A_1^\lambda$ satisfies $\Mod(A_1^\lambda)\to\infty$ as $\lambda\to 0$.
According to Teichm\"{u}ller's modulus extremal theorem \cite[Theorem 4.7, p.\,72]{Ahl10}, as $\lambda\to 0$ we have either $a_\lambda\to 0$, or $a_\lambda\not\to 0$ and $c_2^\lambda \to \infty$. We claim that the second case cannot happen. Indeed, if $c_2^\lambda \to \infty$, then $c_2^\lambda=c_\lambda^+$ and $a_\lambda\to\infty$. Since $c_2^\lambda/a_\lambda=c_\lambda^+/a_\lambda\to 1$ as $a_\lambda\to\infty$, this implies that $\Mod(A_1^\lambda)\not\to\infty$, which is a contradiction. Therefore, we have
$$\lim_{\lambda\to 0}a_\lambda=0.$$
Then $a_\lambda$ can be extended to a holomorphic map $\widetilde{a}_\lambda:\D\to\C\setminus\{-3\pm 2\sqrt{2}\}$ with $\widetilde{a}_\lambda(0)=0$ by removable singularity theorem.
By the openness of $\widetilde{a}_\lambda$, there exists a number $a_1>0$ such that $\widetilde{a}_\lambda(\D)$ contains the segment $[0,a_1)$. In particular, $a_\lambda(\D)$ contains $(0,a_1)$. This finishes the proof.
\end{proof}

An explicit formula of meromorphic functions with exactly one pole having a fixed Herman ring was given in \cite[Proposition 3.4]{DF04} (see also \cite[Proposition 7.19, p.\,240]{BF14}). In particular, the formula is $z\mapsto u z^2 e^z/(z-a)$, where $a$, $u\in\C\setminus\{0\}$. The pre-model before the surgery is $z\mapsto \lambda z e^z$.

If one could prove that for all $\theta\in\mathscr{B}$ and $p\geq 3$, there exist $\kappa$ and $b$ such that $E_{\kappa,b}(z)=\kappa z e^z+b$ contains a $p$-cycle of Siegel disks of rotation number $\theta$ containing the $p$-periodic point $0$, then the statement of Theorem \ref{thm:mero-2} (with $u=(a-b)/(be^b)$ being replaced by $u=u(a,b)$) can be generalized to all $p\geq 1$.

\section{Nested Herman rings}\label{sec:Nested-HR}

The existence of rational maps with a $2$-cycle of nested Herman rings has been proved in \cite[Theorem 5]{Shi87}. The main purpose in this section is to find their formulas.

\begin{proof}[Proof of Theorem \ref{thm:period-2-inverse}]
By the construction of Part (B) in \cite[\S9]{Shi87}, the resulting rational map $h$ after quasiconformal surgery has the following properties:
\begin{itemize}
\item[(i)] $h$ is a rational map of degree $4$ satisfying $h=\tau\circ h\circ \tau^{-1}$, where $\tau(z)=1/\overline{z}$;
\item[(ii)] $h$ has a super-attracting cycle:
\begin{equation}
a\mapsto 0 \mapsto \tau(a)\mapsto \infty\mapsto a,
\end{equation}
where $|a|>1$ and the local degree of $a$ is $3$; and
\item[(iii)] $h$ has a $2$-cycle of Herman rings $\{A_1,\tau(A_1)\}$, where $A_1\subset \Int (\tau(A_1))$.
\end{itemize}
From (i), it follows that $h$ is a Blaschke product of degree $4$ by \cite[Lemma 15.5]{Mil06}. Up to a rotation, we assume that $a>1$. By (ii), $h$ must have the form
\begin{equation}
h(z)=e^{2\pi\ii t}\left(\frac{z-a}{1-a z}\right)^3\,\frac{z-b}{1-\overline{b} z},
\end{equation}
where $0<|b|<1$ and $t\in\R$. From the condition $h(0)=1/a$, we have $b=1/(a^2e^{2\pi\ii t})$.
The proof is finished if we set $r:=1/a$.
\end{proof}

It is much harder to find the formula of a rational map such that it has a $3$-cycle of nested Herman rings $A_1$, $A_2$, $A_3$ which satisfies $A_2\subset \Int (A_1)$ and $A_1\subset \Int (A_3)$. Such examples exist in degree $5$ \cite{Shi02}.

\section{More examples of periodic Herman rings}

In this section we give more explicit examples of rational maps having periodic Herman rings.

\subsection{Herman rings of interlaced case}

In the proof of Theorem \ref{thm:rat-p}, we obtained a $p$-cycle of Herman rings such that the two critical points are associated to exactly one Herman ring. Specifically, suppose that $\theta\in\mathscr{B}$ is chosen such that the boundaries of the Siegel disks $\Delta_1$ and $\Delta$ contain the critical points. Then the Herman rings constructed in \S\ref{sec:non-nested-HR} satisfy that each of the components of $\partial A_1$ contain exactly one critical points while the boundaries of other periodic Herman rings do not contain any critical points.

However, the position of the critical points can be arranged in a crisscross pattern by performing the surgery. For example, for $p\geq 2$, let $P(z)=z^2+\kappa_1 z+\kappa_0$ be a quadratic polynomial such that
\begin{itemize}
\item $P$ has a $p$-cycle of Siegel disks $\{\Delta_i:1\leq i\leq p\}$ of rotation number $\theta$, where $P(\Delta_i)=\Delta_{i+1}$ for $1\leq i\leq p-1$ and $\Delta_1=P(\Delta_p)$;
\item $\Delta_1$ contains the $p$-periodic point $0$; and
\item $\partial\Delta_2$ (or $\partial\Delta_i$ for $2\leq i\leq p$) contains the finite critical point of $P$.
\end{itemize}
Let $\Delta$ be the Siegel disk of $Q(z)=e^{-2\pi\ii\theta}z+z^2$ centered at $0$. As a comparison to the construction in Lemma \ref{lem-rat-pre}, we also paste the exterior of the invariant proper subdisk of $\Delta$ to a small neighborhood of $0$ in $\Delta_1$. Similarly, one can obtain a cubic rational map $f_{a,b}$ (defined in \eqref{equ:f-rat}) having a $p$-cycle of Herman rings $\{A_i:1\leq i\leq p\}$ of rotation number $\theta$, where the internal boundary of $A_1$ and the external boundary of $A_2$, each contains a critical point. We call such a cycle of Herman rings \textit{interlaced}.

An numerical experiment shows that one can choose
\begin{equation}
a=80 \text{\quad and\quad} b\approx46.47151539+3.87122727\ii,
\end{equation}
such that $f_{a,b}$ has a $2$-cycle of interlaced Herman rings. The picture of the corresponding Julia set is very similar to Figure \ref{Fig:Herman}.

\subsection{Herman rings of period $3$}

Theorem \ref{thm:rat-p} guarantees the existence of $a$ and $b$ such that $f_{a,b}$ has a $p$-cycle of Herman rings for any period $p\geq 2$. In particular, for $p=3$, the number $u=u(a,b)$ is determined by the orbit
\begin{equation}
0\mapsto b\mapsto f_{a,b}(b)\mapsto 0,
\end{equation}
A direct calculation shows that $u=u(a,b)$ is the solution of a polynomial with degree $4$ whose coefficients are polynomials of $a$ and $b$.
The solutions cannot be written explicitly, unlike the case $p=2$. However, we can still obtain a picture by numerical experiment (see Figure \ref{Fig:Herman-p=3}).

\begin{figure}[!htpb]
  \setlength{\unitlength}{1mm}
  \includegraphics[width=0.95\textwidth]{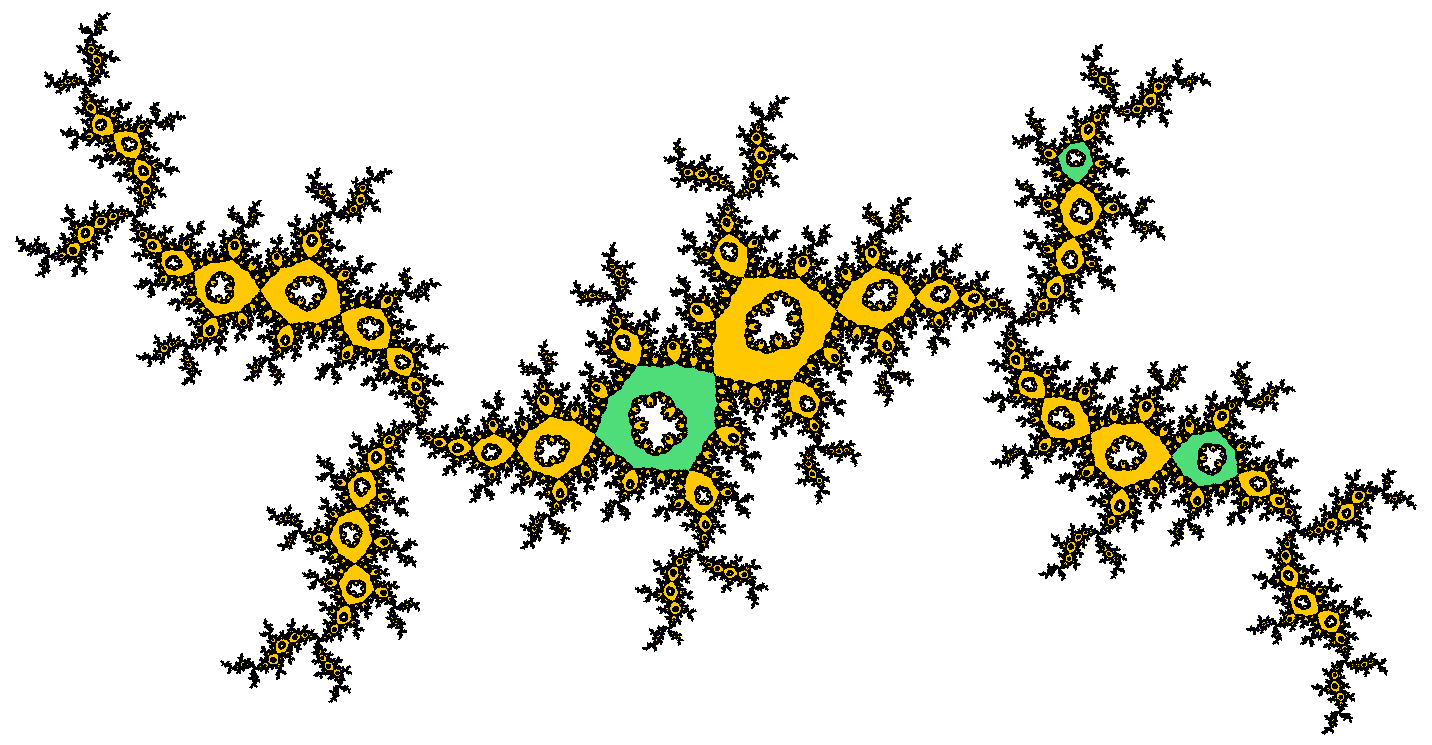}
  \caption{The cubic rational map $f_{a,b}$ with a $3$-cycle of Herman rings (colored green), where $a=4$, $b\approx 12.21173310-1.67440929\ii$ and $u\approx -0.01828231+0.31119224\ii$ satisfying $f_{a,b}^{\circ 3}(0)=0$ are chosen such that the rotation number is $\theta=(\sqrt{5}-1)/2$ (this picture has been rotated). Compare Figure \ref{Fig:Herman}.}
  \label{Fig:Herman-p=3}
\end{figure}

\section{Quadratics in a transcendental family}\label{sec:Siegel}


\subsection{Polynomial-like mappings and Mandelbrot-like set}

We first recall the definition of polynomial-like mappings and state the related theory, which are due to Douady and Hubbard \cite{DH85b}.

\begin{defi}
A triple $(f; U, V)$ is called a \emph{polynomial-like mapping} of degree $d\geq 2$, if $U$ and $V$ are simply connected domains in $\C$ such that $\overline{U}\subset V$, and $f: U\rightarrow V$ is a holomorphic proper mapping of degree $d$. The \emph{filled-in Julia set} $K(f)$ of a polynomial-like mapping $(f;U,V)$ is defined as
$$
K(f)=\{z\in U: f^{\circ n}(z)\in U \text{~for all~}  n\geq 0\}.
$$
The \emph{Julia set} $J(f)$ of $(f;U,V)$ is defined as the boundary of $K(f)$.
In particular, a polynomial-like mapping of degree $2$ is called \textit{quadratic-like}.
\end{defi}

Let $\Lambda$ be a connected complex manifold. Suppose that $\{(f_\lambda;U_\lambda,V_\lambda)\}_{\lambda\in\Lambda}$ is a family of polynomial-like mappings parameterized by $\lambda\in\Lambda$, which satisfies
\begin{itemize}
\item The boundaries of $U_\lambda$ and $V_\lambda$ vary continuously as $\lambda\in\Lambda$ varies; and
\item The map $(\lambda,z)\mapsto f_\lambda(z)$ depends holomorphically on both $\lambda$ and $z$.
\end{itemize}
Then we say that $\{(f_\lambda;U_\lambda,V_\lambda):\lambda\in\Lambda\}$ is a \textit{holomorphic} family of polynomial-like mappings.

The following theorem is the major result in \cite{DH85b}.

\begin{thm}[{Douday-Hubbard}]\label{thm:DH}
Suppose that $\{(f_\lambda;U_\lambda,V_\lambda)\}_{\lambda\in\Lambda}$ is a holomorphic family of quadratic-like mappings. Assume that there exists a closed disk $\mathcal{D}$ contained in $\Lambda$ such that:
\begin{enumerate}
\item For the unique critical point $\omega_{\lambda}$ in $U_{\lambda}$, $f_{\lambda}(\omega_{\lambda})-\omega_{\lambda}$ turns around $0$ once as $\lambda$ turns around $\partial \mathcal{D}$ once; and
\item The critical value $f_\lambda(\omega_\lambda)$ belongs to $V_\lambda\setminus U_\lambda$ for every $\lambda\in\partial\mathcal{D}$.
\end{enumerate}

Then, there exists a subset $\mathcal{M} \subset \mathcal{D}$ which is homeomorphic to the standard Mandelbrot set $\mathbb{M}$ via a map $\lambda\mapsto c(\lambda)$, where $\mathcal{M}$ is given by
\begin{equation}
\mathcal{M}=\left\{\lambda \in \mathcal{D} \mid f_{\lambda}^{\circ n}(\omega_{\lambda}) \in U_{\lambda} \text{ for all } n \geq 0 \right\}.
\end{equation}
Moreover, for each $\lambda \in \mathcal{M}$,  $f_{\lambda}|_{U_{\lambda}}$ is topologically conjugate to $z \mapsto z^{2}+c(\lambda)$
on their corresponding filled-in Julia sets.
\end{thm}

This theorem has many significant applications. In particular, it gives a rigorous mathematical explanation of the appearance of baby Mandelbrot sets in the parameter spaces of various one-dimensional analytic families. In this section, we focus on the following family:
\begin{equation}
E_\lambda(z)=\lambda z^2 e^z, \text{\quad where }\lambda\in\C\setminus\{0\}.
\end{equation}
Note that $E_\lambda$ has an asymptotic value $0$, which is fixed by the super-attracting fixed point $0$ itself. There exists another simple critical point $-2$ which determines the dynamics of $E_\lambda$ essentially. See Figure \ref{Fig:Parameter-FG} for the parameter plane of $E_\lambda$.

\begin{figure}[!htpb]
  \setlength{\unitlength}{1mm}
  \includegraphics[width=0.95\textwidth]{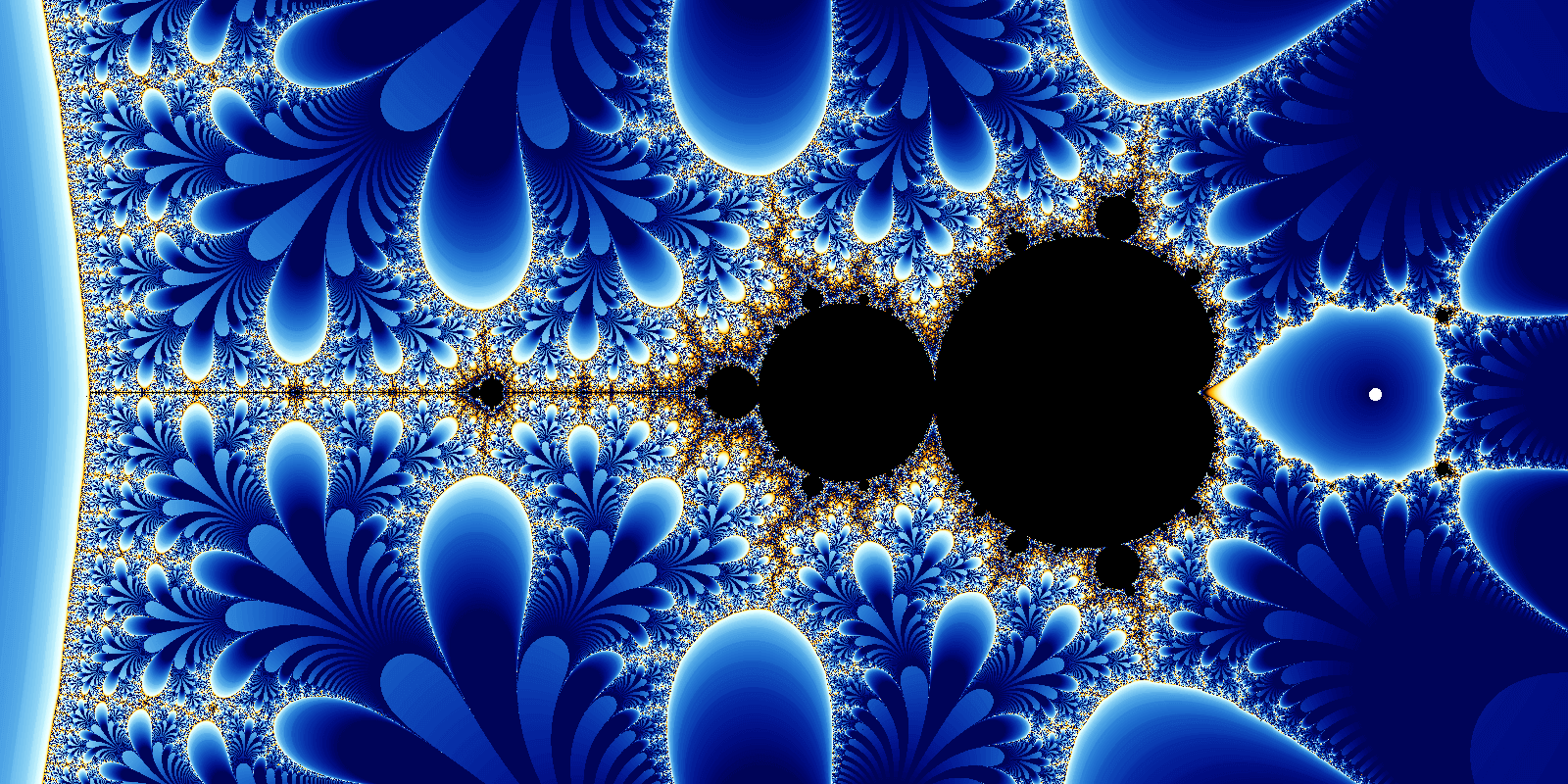}
  \put(-13,29){{\color{white}$0$}}
  \caption{The parameter plane of $E_\lambda(z)=\lambda z^2 e^z$. One can observe a big Mandelbrot-like set in the middle. Figure range: $[-21,3]\times[-6,6]$.  }
  \label{Fig:Parameter-FG}
\end{figure}

\begin{thm}\label{thm:Mandel-FG}
Define
\begin{equation}\label{equ:Lambda}
\begin{split}
\Lambda:=&~\big\{\lambda\in\C:\,\tfrac{2}{5}< |\lambda|<51\big\}\setminus\R^+, \text{\quad and} \\
\mathcal{D}:=&~\big\{\lambda\in\C:\,\tfrac{1}{2}\leq |\lambda|\leq 50 \text{ and } \tfrac{1}{10}\leq \arg\lambda\leq 2\pi-\tfrac{1}{10}\big\}.
\end{split}
\end{equation}
Then we have a holomorphic family of quadratic-like mappings $\{(E_\lambda;U_\lambda,V_\lambda)\}_{\lambda\in\Lambda}$ which satisfies
\begin{enumerate}
\item $U_{\lambda}$ contains exactly one critical point $-2$, and $E_{\lambda}(-2)+2$ turns around zero once as $\lambda$ turns around $\partial \mathcal{D}$ once; and
\item $E_\lambda(-2)\in V_\lambda\setminus U_\lambda$ for every $\lambda\in\partial\mathcal{D}$.
\end{enumerate}
In particular, there exists a Mandelbrot-like set in the interior of $\mathcal{D}$ such that the main cardioid hyperbolic component has period one.
\end{thm}

\begin{proof}[Proof of Proposition \ref{prop:Siegel-p} assuming Theorem \ref{thm:Mandel-FG}]
By Theorems \ref{thm:DH} and \ref{thm:Mandel-FG},  there exists a homeomorphism $\chi:\lambda\mapsto c=c(\lambda)$
between a subset $\mathcal{M}$ of $\mathcal{D}$ and the Mandelbrot set $\mathbb{M}$ such that, for every $\lambda \in \mathcal{M}$, the map $E_{\lambda}:U_{\lambda}\to V_\lambda$ is topologically conjugate to the quadratic polynomial $P_c(z)=z^{2}+c$ on their corresponding filled-in Julia sets.

For any given Brjuno number $\theta$ and any integer $p\geq 1$, there exists $c_0\in\mathbb{M}$ such that $P_{c_0}$ contains a $p$-cycle of Siegel disks with rotation number $\theta$. This implies that $E_{\chi^{-1}(c_0)}$ contains a $p$-cycle of Siegel disks with rotation number $\theta$.
\end{proof}

\subsection{Proof of Theorem \ref{thm:Mandel-FG}}


In the following, the argument function $\arg z$ is assumed to take values in $[0,2\pi)$ for any $z\in\C\setminus\{ 0\}$. We first give a decomposition of dynamical plane of $E_\lambda=\lambda z^2 e^z$ for all $\lambda\in\C\setminus[0,+\infty)$. Note that we have $\arg \lambda\in(0,2\pi)$.

Consider the preimages of positive real line $\R^+=(0,+\infty)$ under $E_\lambda(z)$ and define
\begin{equation}
\sigma_k^\lambda:=\{z\in\C:\im z = (2 k+2)\pi-\arg\lambda-\arg (z^2)\},
\end{equation}
where $k\in\Z$. Then we have $E_\lambda^{-1}(\R^+)=\bigcup_{k\in\Z}\sigma_k^\lambda$.
Obviously $\sigma_k^\lambda\cap\R=\emptyset$ for any $k\in\Z$. Define open horizontal strips:
\begin{equation}
S_k^\lambda:=\{z\in\C:2k\pi-\arg\lambda<\im z<(2k+2)\pi-\arg\lambda\},
\end{equation}
where $k\in\Z$. Let $\BH^{\pm}=\{z\in\C:\pm\,\im z>0\}$ be the upper and lower half planes respectively.

\begin{lem}\label{lem:S_k}
For all $k\in\Z$, $\sigma_k^\lambda\subset S_k^\lambda$. In particular,
\begin{enumerate}
\item If $k\geq 1$, then $\sigma_k^\lambda=\{x+\sigma_k^\lambda(x)\ii:x\in\R\}$, where $\sigma_k^\lambda(x)$ is real-analytic, strictly increasing and satisfies
\begin{equation}
\lim_{x\to-\infty}\sigma_k^\lambda(x)=2k\pi-\arg\lambda \text{\quad and\quad} \lim_{x\to+\infty}\sigma_k^\lambda(x)=(2k+2)\pi-\arg\lambda.
\end{equation}
\item If $k\leq -1$, then $\sigma_k^\lambda=\{x+\sigma_k^\lambda(x)\ii:x\in\R\}$, where $\sigma_k^\lambda(x)$ is real-analytic, strictly decreasing and satisfies
\begin{equation}
\lim_{x\to-\infty}\sigma_k^\lambda(x)=(2k+2)\pi-\arg\lambda \text{\quad and\quad} \lim_{x\to+\infty}\sigma_k^\lambda(x)= 2k\pi-\arg\lambda.
\end{equation}
\item If $k=0$, then $\sigma_0^\lambda=\sigma_{0,+}^\lambda\cup\sigma_{0,-}^\lambda$, where $\sigma_{0,\pm}^\lambda$ are real-analytic curves in $S_0^\lambda\cap\BH^{\pm}$ respectively. Moreover,  one end point of $\sigma_{0,\pm}^\lambda$ is $0$,  $\sigma_{0,\pm}^\lambda\cap\{z\in\C:\re z>0\}=\{x+\sigma_{0,\pm}^\lambda(x)\ii:x>0\}$, $\sigma_{0,+}^\lambda(x)$ and $\sigma_{0,-}^\lambda(x)$ are strictly increasing and decreasing in $(0,+\infty)$ respectively, and
\begin{equation}
\lim_{x\to+\infty}\sigma_{0,+}^\lambda(x)=2\pi-\arg\lambda \text{\quad and\quad} \lim_{x\to+\infty}\sigma_{0,-}^\lambda(x)=-\arg\lambda.
\end{equation}
\end{enumerate}
\end{lem}

See Figure \ref{Fig:Dynam-FG}. From Lemma \ref{lem:S_k} we know that the two boundary components of $S_k^\lambda$ are asymptotic lines of $\sigma_k^\lambda$ for all $k\in\Z$.

\begin{figure}[!htpb]
  \setlength{\unitlength}{1mm}
  \includegraphics[width=0.85\textwidth]{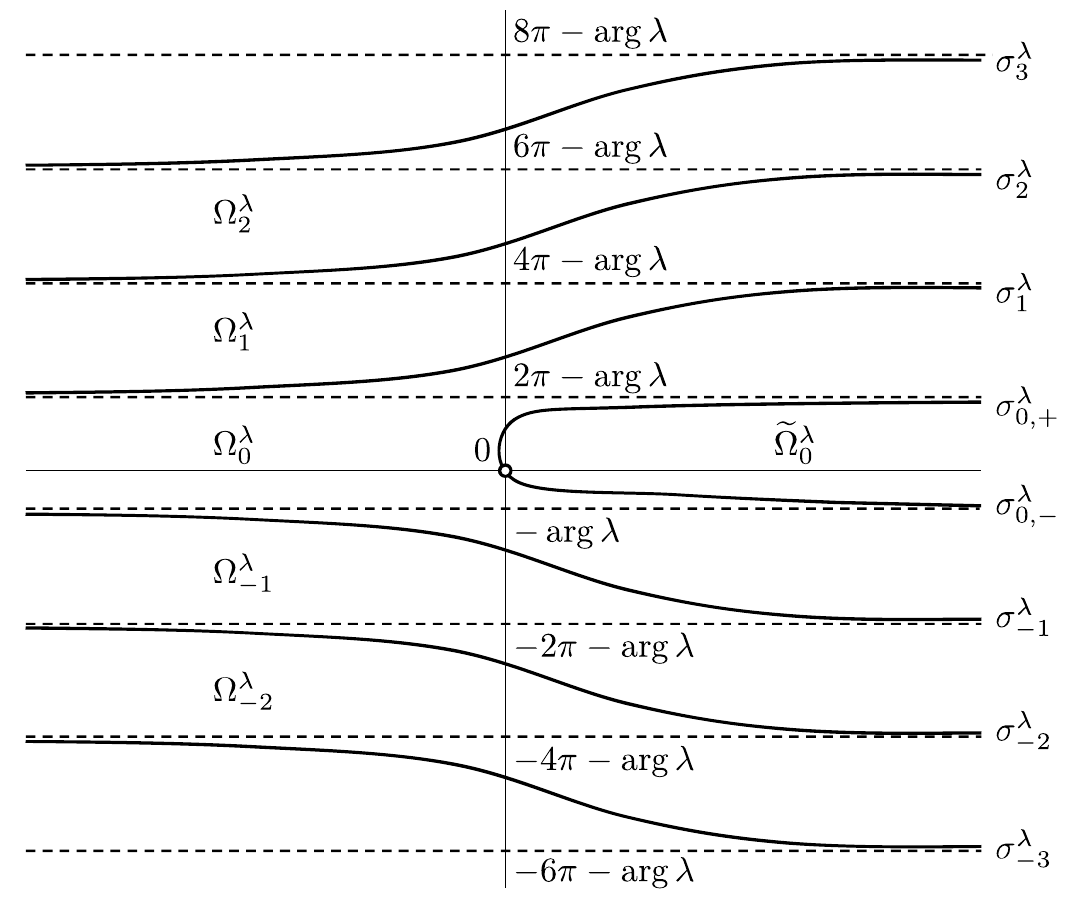}
  \caption{The decomposition of the dynamical plane of $E_\lambda$ by the curves $\{\sigma_k^\lambda:k\in\Z\}$, which are preimages of $\R^+$ and contained in the horizontal strips $\{S_k^\lambda:k\in\Z\}$ respectively. The set $\sigma_0^\lambda$ is special, which consists of two curves $\sigma_{0,+}^\lambda$ and $\sigma_{0,-}^\lambda$ and both of them land at $0$.}
  \label{Fig:Dynam-FG}
\end{figure}

\begin{proof}
(a) If $k\geq 1$, then $(2 k+2)\pi-\arg\lambda-\arg (z^2)>0$ since $\arg\lambda\in(0,2\pi)$ and $\arg(z^2)\in[0,2\pi)$. This implies that $\sigma_k^\lambda\subset\BH^+$. For any $\theta\in(0,\pi)$, the intersection of the radial $r_\theta$ and the line $L_{k,\theta,+}^\lambda\subset\BH^+$ is a singleton $z_{k,\theta}^\lambda\in\sigma_k^\lambda$, where
\begin{equation}
\begin{split}
r_\theta=&~\{z\in\C:\arg z=\theta\} \text{\quad and} \\
L_{k,\theta,+}^\lambda=&~\{z\in\C:\im z=(2k+2)\pi-\arg\lambda-2\theta\}.
\end{split}
\end{equation}
As $\theta$ varies from $0$ to $\pi$, the imaginary part of $z_{k,\theta}^\lambda$ is strictly decreasing. In particular, we have
\begin{equation}
\lim_{\theta\to 0^+}z_{k,\theta}^\lambda=(2k+2)\pi-\arg\lambda \text{\quad and\quad} \lim_{\theta\to \pi^-}z_{k,\theta}^\lambda=2k\pi-\arg\lambda.
\end{equation}
Hence we have $\sigma_k^\lambda\subset S_k^\lambda$ and $\sigma_k^\lambda$ is a real-analytic curve since $\R^+$ does not contain the critical value $E(-2)=4\lambda/e^2$.

\medskip

(b) If $k\leq -1$, then $(2 k+2)\pi-\arg\lambda-\arg (z^2)<0$. This implies that $\sigma_k^\lambda\subset\BH^-$. For any $\theta\in(\pi,2\pi)$, $r_\theta\cap L_{k,\theta,-}^\lambda$ is a singleton $z_{k,\theta}^\lambda\in\sigma_k^\lambda$, where
\begin{equation}
L_{k,\theta,-}^\lambda=\{z\in\C:\im z=(2k+4)\pi-\arg\lambda-2\theta\}.
\end{equation}
 As $\theta$ varies from $\pi$ to $2\pi$, the imaginary part of $z_{k,\theta}^\lambda$ is strictly decreasing. In particular, we have
\begin{equation}
\lim_{\theta\to \pi^+}z_{k,\theta}^\lambda=(2k+2)\pi-\arg\lambda \text{\quad and\quad} \lim_{\theta\to 2\pi^-}z_{k,\theta}^\lambda=2k\pi-\arg\lambda.
\end{equation}
This implies that $\sigma_k^\lambda\subset S_k^\lambda$ and $\sigma_k^\lambda$ is also real-analytic.

\medskip

(c) If $k=0$, then the sign of $2\pi-\arg\lambda-\arg (z^2)$ depends on $\lambda$ and $z$. There are following two cases:
\begin{itemize}
\item If $\theta\in(0,\pi)$, then $r_\theta\cap L_{0,\theta,+}^\lambda\neq\emptyset$ if and only if $\theta\in (0,\pi-\frac{\arg \lambda}{2})$;
\item If $\theta\in(\pi,2\pi)$, then $r_\theta\cap L_{0,\theta,-}^\lambda\neq\emptyset$ if and only if $\theta\in (2\pi-\frac{\arg \lambda}{2},2\pi)$.
\end{itemize}
Moreover, when $r_\theta\cap L_{0,\theta,+}^\lambda$ (or $r_\theta\cap L_{0,\theta,-}^\lambda$) is non-empty, then it is a singleton, which is denote by $z_{0,\theta}^\lambda$.
As $\theta$ varies from $0$ to $\pi-\frac{\arg \lambda}{2}$, or from $2\pi-\frac{\arg \lambda}{2}$ to $2\pi$, the imaginary part of $z_{0,\theta}^\lambda$ is strictly decreasing. We have
\begin{equation}
\lim_{\theta\to 0^+}z_{0,\theta}^\lambda=2\pi-\arg\lambda \text{\quad and\quad} \lim_{\theta\to 2\pi^-}z_{0,\theta}^\lambda=-\arg\lambda.
\end{equation}
Moreover, $z_{0,\theta}^\lambda\to 0$ as $\theta\to\big(\pi-\frac{\arg \lambda}{2}\big)^-$ or $\theta\to\big(2\pi-\frac{\arg \lambda}{2}\big)^+$.
This implies that $\sigma_0^\lambda=\sigma_{0,+}^\lambda\cup\sigma_{0,-}^\lambda\subset S_0^\lambda$ and $\sigma_{0,+}^\lambda$, $\sigma_{0,-}^\lambda$ are real-analytic curves.
\end{proof}

Note that $E_\lambda:\sigma_k^\lambda \to \R^+$ with $k\neq 0$ and $E_\lambda:\sigma_{0,\pm}^\lambda \to \R^+$  are homeomorphisms.
The set $\sigma_0^\lambda\cup\{0\}=\sigma_{0,+}^\lambda\cup\{0\}\cup \sigma_{0,-}^\lambda$ is an analytic curve of shape $C$ (see Figure \ref{Fig:Dynam-FG}). For $k\geq 1$, we use $\Omega_k^\lambda$ to denote the simply connected domain bounded by $\sigma_k^\lambda$ and $\sigma_{k+1}^\lambda$. For $k\leq -1$, we use $\Omega_k^\lambda$ to denote the simply connected domain bounded by $\sigma_k^\lambda$ and $\sigma_{k-1}^\lambda$. For $k=0$, we use $\Omega_0^\lambda$ to denote the simply connected domain bounded by $\sigma_1^\lambda$, $\sigma_{-1}^\lambda$ and $\sigma_0^\lambda\cup\{0\}$. Let $\widetilde{\Omega}_0^\lambda$ be the component of $\C\setminus(\sigma_0^\lambda\cup\{0\})$ containing $\R^+$.

\medskip
Besides the critical point $0$, the remaining critical point $-2$ is contained in $\Omega_0^\lambda$. The following result follows from Lemma \ref{lem:S_k} and Riemann-Hurwitz's formula:

\begin{cor}\label{cor:degree}
We have
\begin{itemize}
\item $E_\lambda:\Omega_k^\lambda\to\C\setminus [0,+\infty)$ is conformal, for all $k\neq 0$;
\item $E_\lambda:\widetilde{\Omega}_0^\lambda\to\C\setminus [0,+\infty)$ is conformal; and
\item $E_\lambda:\Omega_0^\lambda\to\C\setminus [0,+\infty)$ is proper of degree $2$.
\end{itemize}
\end{cor}

\begin{proof}[Proof of Theorem \ref{thm:Mandel-FG}]
For $\lambda\in\Lambda=\big\{\lambda\in\C:\,\tfrac{2}{5}< |\lambda|<51\big\}\setminus\R^+$, we define
\begin{equation}
V_\lambda:=\{z\in\C: \tfrac{1}{16|\lambda|}<|z|<30\}\setminus\R^+.
\end{equation}
Let $U_\lambda:=E_\lambda^{-1}(V_\lambda)\cap\Omega_0^\lambda$. See Figure \ref{Fig:Poly-like-FG}. We claim that $\{(E_\lambda;U_\lambda,V_\lambda)\}_{\lambda\in\Lambda}$ is a holomorphic family of quadratic-like mappings with critical point $-2\in U_\lambda$. Since $\Omega_0^\lambda\cap[0,+\infty)=\emptyset$, by Corollary \ref{cor:degree}, it is sufficient to prove that for all $\lambda\in\Lambda$,
\begin{itemize}
\item The critical value $E_\lambda(-2)=4\lambda/e^2\in V_\lambda$; and
\item $U_\lambda$ is contained in $\{z\in\C: \tfrac{1}{6|\lambda|}<|z|<25\}$.
\end{itemize}
Obviously we have $E_\lambda(-2)\not\in\R^+$ if $\lambda\not\in\R^+$. If $\frac{2}{5}<|\lambda|<51$, then
\begin{equation}
|E_\lambda(-2)|=\frac{4|\lambda|}{e^2}<\frac{204}{2.7^2}<30,
\end{equation}
and $|E_\lambda(-2)|>\frac{1}{16|\lambda|}$ since $|\lambda|>\frac{2}{5}>\frac{e}{8}$. This implies that $E_\lambda(-2)\in V_\lambda$ for all $\lambda\in\Lambda$.

\begin{figure}[!htpb]
  \setlength{\unitlength}{1mm}
  \includegraphics[width=0.85\textwidth]{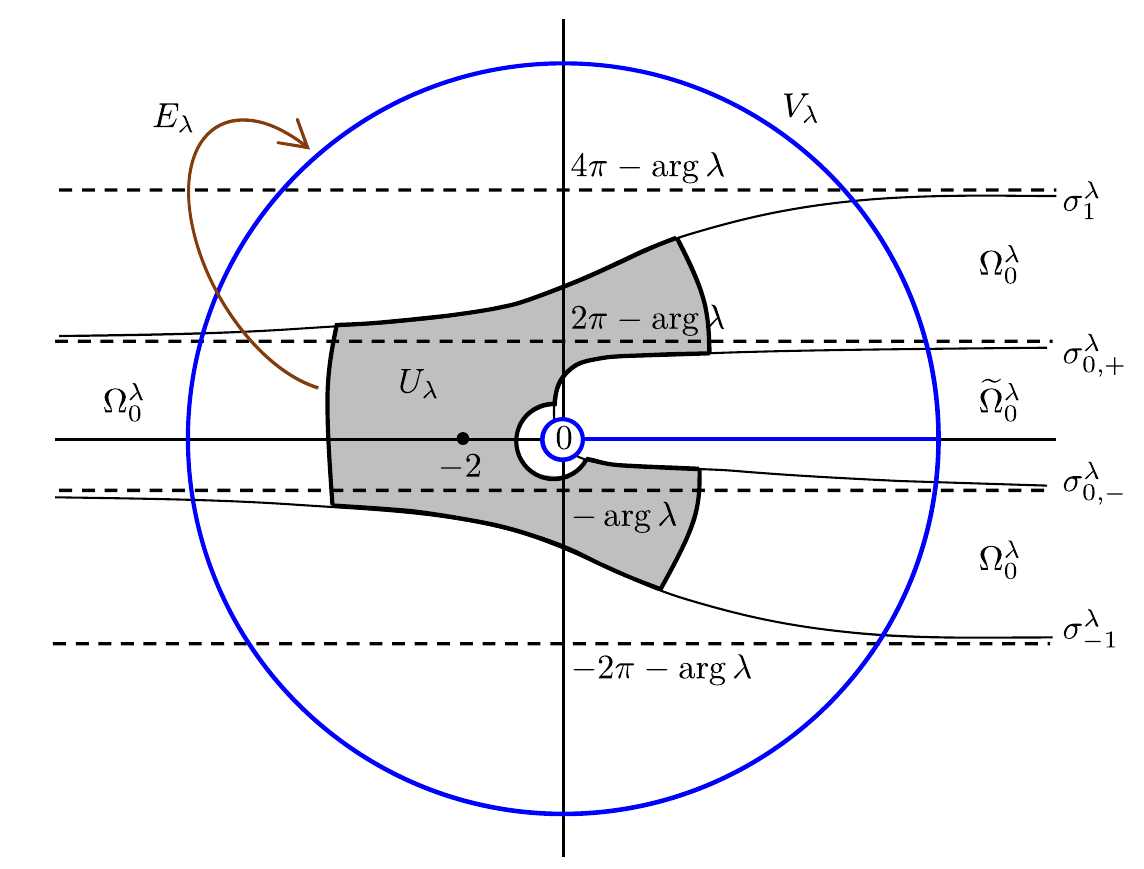}
  \caption{The sketch of the construction of quadratic-like mappings $(E_\lambda;U_\lambda,V_\lambda)$.}
  \label{Fig:Poly-like-FG}
\end{figure}

By Lemma \ref{lem:S_k}, it follows that $U_\lambda\subset S_{-1}^\lambda\cup S_0^\lambda\cup S_1^\lambda$ and hence $|\im z|<4\pi$ for any $z\in U_\lambda$. We claim that $U_\lambda\subset\{z\in\C:-20<\re z<5\}$. Indeed, if $\re z\geq 5$, then
\begin{equation}
|E_\lambda(z)|=|\lambda|\cdot|z|^2\cdot e^{\re z}\geq \tfrac{2}{5}\cdot 5^2 \cdot e^5>30.
\end{equation}
Note that
\begin{equation}
20^2+(4\pi)^2<20^2+13^2=569<625=25^2<30^2,
\end{equation}
and $x\mapsto x^2 e^x$ is increasing on $(-\infty,-2]$. If $\re z\leq -20$ and $|\im z|<4\pi$, then
\begin{equation}
|E_\lambda(z)|\leq 51\cdot \big(20^2+(4\pi)^2\big) \,e^{-20}<50\cdot 5^4\cdot \left(\frac{2}{5}\right)^{20}<\frac{1}{2^6\cdot 5^2}<\frac{1}{16|\lambda|}.
\end{equation}
Here we use the fact that $51\cdot 569<50\cdot 625$ and $2^{27}<5^{12}$. Hence $U_\lambda\subset\{z\in\C: |z|<25\}$. If $|z|\leq \frac{1}{6|\lambda|}$, then
\begin{equation}
|E_\lambda(z)|\leq |\lambda|\cdot\left(\frac{1}{6|\lambda|}\right)^2 e^{\frac{1}{6|\lambda|}}\leq\frac{1}{36|\lambda|} e^{5/12}<\frac{1}{18|\lambda|}<\frac{1}{16|\lambda|}.
\end{equation}
Hence $U_\lambda\subset\{z\in\C: \tfrac{1}{6|\lambda|}<|z|<25\}$. This ends the proof that $\{(E_\lambda;U_\lambda,V_\lambda)\}_{\lambda\in\Lambda}$ is a holomorphic family of quadratic-like mappings with $-2\in U_\lambda$.

\medskip
The boundary of $\mathcal{D}=\big\{\lambda\in\C:\,\tfrac{1}{2}\leq |\lambda|\leq 50 \text{ and } \tfrac{1}{10}\leq \arg\lambda\leq 2\pi-\tfrac{1}{10}\big\}$ is the union of following $4$ curves:
\begin{equation}
\begin{split}
\gamma_1=&~\{\lambda\in\C:|\lambda|=\tfrac{1}{2}, ~\arg\lambda\in[\tfrac{1}{10},2\pi-\tfrac{1}{10}]\}, \\
\gamma_2=&~\{\lambda\in\C:|\lambda|=50, ~\arg\lambda\in[\tfrac{1}{10},2\pi-\tfrac{1}{10}]\}, \\
\gamma_3=&~\{\lambda\in\C:|\lambda|\in[\tfrac{1}{2},50], ~\arg\lambda=\tfrac{1}{10}\}, \text{ and}\\
\gamma_4=&~\{\lambda\in\C:|\lambda|\in[\tfrac{1}{2},50], ~\arg\lambda=2\pi-\tfrac{1}{10}\}.
\end{split}
\end{equation}
Obviously, $E_{\lambda}(-2)+2=4\lambda/e^2+2$ turns around $0$ once as $\lambda$ turns around $\partial \mathcal{D}$ once. In the following we prove that
$4\lambda/e^2\in V_\lambda\setminus U_\lambda$ for every $\lambda\in\partial\mathcal{D}$. Since $4\lambda/e^2\in V_\lambda$ has been proved for all $\lambda\in\Lambda$, it suffices to show that $4\lambda/e^2\not\in U_\lambda$ for all $\lambda\in\partial\mathcal{D}$.

\medskip
For simplicity, we denote by $v_\lambda:=4\lambda/e^2$. Recall that $U_\lambda$ is contained in $\{z\in\C: \tfrac{1}{6|\lambda|}<|z|<25\}$. If $\lambda\in\gamma_1$, then
\begin{equation}
|v_\lambda|=\frac{2}{e^2}<\frac{2}{(\tfrac{5}{2})^2}<\frac{1}{3}=\frac{1}{6|\lambda|}.
\end{equation}
This implies that $v_\lambda\not\in U_\lambda$ for $\lambda\in\gamma_1$.

Note that $e^2<2.8^2<8$. If $\lambda\in\gamma_2$, then $|v_\lambda|=200/e^2>25$.
This implies that $v_\lambda\not\in U_\lambda$ for $\lambda\in\gamma_2$.

If $\lambda\in\gamma_3$, we claim that $v_\lambda\in\widetilde{\Omega}_0^\lambda$. Denote by $\theta=\frac{1}{10}$. According to the proof of Lemma \ref{lem:S_k}(c), it suffices to prove that $\im z_{0,\theta}^\lambda>\im v_\lambda$. This is true since
\begin{equation}
\begin{split}
&~(2\pi-\arg\lambda-2\arg z_{0,\theta}^\lambda)-\frac{4}{e^2}|\lambda|\sin\theta>2\pi-3\theta-\frac{4}{e^2}|\lambda|\theta \\
\geq &~2\pi-\Big(3+\frac{200}{e^2}\Big)\cdot\frac{1}{10}>2\pi-\frac{53}{10}>0.
\end{split}
\end{equation}
Since $\Omega_0^\lambda\cap \widetilde{\Omega}_0^\lambda=\emptyset$, we conclude that $v_\lambda\not\in U_\lambda$ for $\lambda\in\gamma_3$.
By the symmetry of the parameters, we have $v_\lambda\not\in U_\lambda$ for $\lambda\in\gamma_4$. The proof is complete.
\end{proof}

One may extend the result of Theorem \ref{thm:Mandel-FG} to $E_{\lambda,m}(z)=\lambda z^m e^z$ for all $m\geq 2$. In particular, if $m$ is even, we consider the preimages of positive real axis, while if $m\geq 3$ is odd, we consider the preimages of negative real axis. One can construct a holomorphic family of quadratic-like mappings $(E_{\lambda,m};U_{\lambda,m},V_{\lambda,m})$ and prove that there exists a copy of the Mandelbrot set in the parameter space of $E_{\lambda,m}$, which has been observed by computer experiments \cite{FG07}.

\section*{Acknowledgements}
The author would like to thank Arnaud Ch\'{e}ritat for providing an algorithm to draw Figure \ref{Fig:Herman-mero} and Lasse Rempe for very helpful comments.
He is also very grateful to the referee for very insightful and detailed comments, suggestions and corrections.
This work was supported by NSFC (grant No.\,12071210) and NSF of Jiangsu Province (grant No.\,BK20191246).


\bibliographystyle{amsalpha}
\bibliography{E:/Latex-model/Ref1}

\providecommand{\bysame}{\leavevmode\hbox to3em{\hrulefill}\thinspace}
\providecommand{\MR}{\relax\ifhmode\unskip\space\fi MR }
\providecommand{\MRhref}[2]{%
  \href{http://www.ams.org/mathscinet-getitem?mr=#1}{#2}
}
\providecommand{\href}[2]{#2}
\begin{thebibliography}{BFGH05}

\bibitem[Ahl10]{Ahl10}
L.~V. Ahlfors, \emph{Conformal invariants: topics in geometric function
  theory}, AMS Chelsea Publishing, Providence, RI, 2010.

\bibitem[BBM18]{BBM18}
A.~Bonifant, X.~Buff, and J.~Milnor, \emph{Antipode preserving cubic maps: the
  fjord theorem}, Proc. Lond. Math. Soc. (3) \textbf{116} (2018), no.~3,
  670--728.

\bibitem[BD98]{BD98}
I.~N. Baker and P.~Dom\'{\i}nguez, \emph{Analytic self-maps of the punctured
  plane}, Complex Variables Theory Appl. \textbf{37} (1998), no.~1-4, 67--91.

\bibitem[BF14]{BF14}
B.~Branner and N.~Fagella, \emph{Quasiconformal surgery in holomorphic
  dynamics, \textup{with contributions by X. Buff, S. Bullett, A. L. Epstein,
  P. Ha\"{\i}ssinsky, C. Henriksen, C. L. Petersen, K. M. Pilgrim, L. Tan and
  M. Yampolsky}}, Cambridge Studies in Advanced Mathematics, vol. 141,
  Cambridge University Press, Cambridge, 2014.

\bibitem[BFGH05]{BFGH05}
X.~Buff, N.~Fagella, L.~Geyer, and C.~Henriksen, \emph{Herman rings and
  {A}rnold disks}, J. London Math. Soc. (2) \textbf{72} (2005), no.~3,
  689--716.

\bibitem[BKL91]{BKL91b}
I.~N. Baker, J.~Kotus, and Y.~L\"{u}, \emph{Iterates of meromorphic functions
  {III}: {P}reperiodic domains}, Ergodic Theory Dynam. Systems \textbf{11}
  (1991), no.~4, 603--618.

\bibitem[DF04]{DF04}
P.~Dom\'{\i}nguez and N.~Fagella, \emph{Existence of {H}erman rings for
  meromorphic functions}, Complex Var. Theory Appl. \textbf{49} (2004), no.~12,
  851--870.

\bibitem[DH85]{DH85b}
A.~Douady and J.~H. Hubbard, \emph{On the dynamics of polynomial-like
  mappings}, Ann. Sci. \'{E}cole Norm. Sup. (4) \textbf{18} (1985), no.~2,
  287--343.

\bibitem[Fag95]{Fag95}
N.~Fagella, \emph{Limiting dynamics for the complex standard family}, Internat.
  J. Bifur. Chaos Appl. Sci. Engrg. \textbf{5} (1995), no.~3, 673--699.

\bibitem[Fag99]{Fag99}
\bysame, \emph{Dynamics of the complex standard family}, J. Math. Anal. Appl.
  \textbf{229} (1999), no.~1, 1--31.

\bibitem[FG03a]{FG03a}
N.~Fagella and A.~Garijo, \emph{Capture zones of the family of functions
  {$\lambda z^m\exp(z)$}}, Internat. J. Bifur. Chaos Appl. Sci. Engrg.
  \textbf{13} (2003), no.~9, 2623--2640.

\bibitem[FG03b]{FG03}
N.~Fagella and L.~Geyer, \emph{Surgery on {H}erman rings of the complex
  standard family}, Ergodic Theory Dynam. Systems \textbf{23} (2003), no.~2,
  493--508.

\bibitem[FG07]{FG07}
N.~Fagella and A.~Garijo, \emph{The parameter planes of {$\lambda z^m\exp(z)$}
  for {$m\geq 2$}}, Comm. Math. Phys. \textbf{273} (2007), no.~3, 755--783.

\bibitem[FH06]{FH06}
N.~Fagella and C.~Henriksen, \emph{Arnold disks and the moduli of {H}erman
  rings of the complex standard family}, Dynamics on the {R}iemann sphere, Eur.
  Math. Soc., Z\"{u}rich, 2006, pp.~161--182.

\bibitem[FP12]{FP12}
N.~Fagella and J.~Peter, \emph{On the configuration of {H}erman rings of
  meromorphic functions}, J. Math. Anal. Appl. \textbf{394} (2012), no.~2,
  458--467.

\bibitem[FSV04]{FSV04}
N.~Fagella, T.~M. Seara, and J.~Villanueva, \emph{Asymptotic size of {H}erman
  rings of the complex standard family by quantitative quasiconformal surgery},
  Ergodic Theory Dynam. Systems \textbf{24} (2004), no.~3, 735--766.

\bibitem[Gey01]{Gey01}
L.~Geyer, \emph{Siegel discs, {H}erman rings and the {A}rnold family}, Trans.
  Amer. Math. Soc. \textbf{353} (2001), no.~9, 3661--3683.

\bibitem[Hen14]{Hen14}
C.~Henriksen, \emph{Animations of {H}erman rings},
  \href{http://www2.mat.dtu.dk/people/Christian.Henriksen/arnolddisk.html}{http://www2.mat.dtu.dk/people/Chris
  tian.Henriksen/arnolddisk.html}, 2014.

\bibitem[Her79]{Her79}
M.~R. Herman, \emph{Sur la conjugaison diff\'{e}rentiable des
  diff\'{e}omorphismes du cercle \`a des rotations}, Inst. Hautes \'{E}tudes
  Sci. Publ. Math. (1979), no.~49, 5--233.

\bibitem[Her85]{Her85}
\bysame, \emph{Are there critical points on the boundaries of singular
  domains?}, Comm. Math. Phys. \textbf{99} (1985), no.~4, 593--612.

\bibitem[HK04]{HK04}
J.~Hawkins and L.~Koss, \emph{Parametrized dynamics of the {W}eierstrass
  elliptic function}, Conform. Geom. Dyn. \textbf{8} (2004), 1--35.

\bibitem[HR17]{HR17}
J.~Hawkins and M.~Randolph, \emph{An experimental view of {H}erman rings for
  dianalytic maps of $\mathbb{RP}^2$}, arXiv: 1706.09980, 2017.

\bibitem[HX19]{HX19}
J.~Hu and Y.~Xiao, \emph{No {H}erman rings for regularly ramified rational
  maps}, Proc. Amer. Math. Soc. \textbf{147} (2019), no.~4, 1587--1596.

\bibitem[Kat17]{Kat17}
K.~Katagata, \emph{Entire functions whose {J}ulia sets include any finitely
  many copies of quadratic {J}ulia sets}, Nonlinearity \textbf{30} (2017),
  no.~6, 2360--2380.

\bibitem[Kre01]{Kre01}
D.~Kremer, \emph{Some investigations on the dynamics of the family {$\lambda
  z\exp(z)$}}, Complex Variables Theory Appl. \textbf{45} (2001), no.~4,
  355--370.

\bibitem[Ma{\~{n}}85]{Man85}
R.~Ma{\~{n}}{\'{e}}, \emph{On the instability of {H}erman rings}, Invent. Math.
  \textbf{81} (1985), no.~3, 459--471.

\bibitem[Mil00]{Mil00b}
J.~Milnor, \emph{On rational maps with two critical points}, Experiment. Math.
  \textbf{9} (2000), no.~4, 481--522.

\bibitem[Mil06]{Mil06}
\bysame, \emph{Dynamics in one complex variable}, third ed., Annals of
  Mathematics Studies, vol. 160, Princeton University Press, Princeton, NJ,
  2006.

\bibitem[Nay16]{Nay16}
T.~Nayak, \emph{Herman rings of meromorphic maps with an omitted value}, Proc.
  Amer. Math. Soc. \textbf{144} (2016), no.~2, 587--597.

\bibitem[Roc20]{Roc20}
M.~M. Rocha, \emph{Herman rings of elliptic functions}, Arnold Math. J.
  \textbf{6} (2020), no.~3-4, 551--570.

\bibitem[RS09]{RS09}
L.~Rempe and D.~Schleicher, \emph{Bifurcations in the space of exponential
  maps}, Invent. Math. \textbf{175} (2009), no.~1, 103--135.

\bibitem[Shi86]{Shi86}
M.~Shishikura, \emph{Surgery of complex analytic dynamical systems}, Dynamical
  systems and nonlinear oscillations ({K}yoto, 1985), World Sci. Adv. Ser.
  Dynam. Systems, vol.~1, World Sci. Publishing, Singapore, 1986, pp.~93--105.

\bibitem[Shi87]{Shi87}
\bysame, \emph{On the quasiconformal surgery of rational functions}, Ann. Sci.
  \'{E}cole Norm. Sup. (4) \textbf{20} (1987), no.~1, 1--29.

\bibitem[Shi89]{Shi89}
\bysame, \emph{Trees associated with the configuration of {H}erman rings},
  Ergodic Theory Dynam. Systems \textbf{9} (1989), no.~3, 543--560.

\bibitem[Shi02]{Shi02}
\bysame, \emph{A new tree associated with {H}erman rings}, Complex Dynamics and
  Related Fields,
  \href{http://www.kurims.kyoto-u.ac.jp/~kyodo/kokyuroku/contents/pdf/1269-8.pdf}{RIMS
  K{\^{o}}ky{\^{u}}roku}, vol. 1269, 2002, pp.~74--92.

\bibitem[Shi14]{Shi14}
\bysame, \emph{Tropical complex dynamics}, in Holomorphic Dynamics in One and
  Several Variables,
  \href{https://www.math.stonybrook.edu/koreadyn/Talks/Shishikura.pdf}{https://www.math.stonybrook.edu/koreadyn/Talks/Shishikura.pdf},
  2014.

\bibitem[Yan17]{Yan17}
F.~Yang, \emph{Rational maps without {H}erman rings}, Proc. Amer. Math. Soc.
  \textbf{145} (2017), no.~4, 1649--1659.

\bibitem[Yoc95]{Yoc95}
J.-C. Yoccoz, \emph{Th\'{e}or\`eme de {S}iegel, nombres de {B}runo et
  polyn\^{o}mes quadratiques}, Ast\'{e}risque (1995), no.~231, 3--88.

\bibitem[Yoc02]{Yoc02}
\bysame, \emph{Analytic linearization of circle diffeomorphisms}, Dynamical
  systems and small divisors ({C}etraro, 1998), Lecture Notes in Math., vol.
  1784, Springer, Berlin, 2002, pp.~125--173.

\bibitem[Zhe00]{Zhe00}
J.~Zheng, \emph{Remarks on {H}erman rings of transcendental meromorphic
  functions}, Indian J. Pure Appl. Math. \textbf{31} (2000), no.~7, 747--751.

\end{thebibliography}

\end{document}